\documentclass{amsart}

\usepackage{graphicx}
\usepackage{libertine}
\usepackage{wasysym}

\newtheorem{theorem}{Theorem}[section]
\newtheorem{lemma}[theorem]{Lemma}
\newtheorem{corollary}[theorem]{Corollary}
\newtheorem{prop}[theorem]{Proposition}

\newtheorem{ap}[theorem]{A Priori Estimate}
\newtheorem{assumption}[theorem]{Assumption}
\newtheorem{notation}[theorem]{Notation}

\theoremstyle{definition}
\newtheorem{definition}[theorem]{Definition}

\theoremstyle{remark}
\newtheorem{remark}[theorem]{Remark}

\numberwithin{equation}{section}

\usepackage[top=0.75in, bottom=0.75in, left=0.75in, right=0.75in]{geometry}

\usepackage{mathrsfs}

\usepackage[english]{babel}
\usepackage[utf8x]{inputenc}
\usepackage{amsmath}
\usepackage{graphicx}
\usepackage{amssymb}
\usepackage{amsthm}
\usepackage{fancyhdr}
\usepackage{amsmath}
\usepackage{amsxtra}
\usepackage{amscd}
\usepackage{amsthm}
\usepackage{amsfonts}
\usepackage{amssymb}
\usepackage{eucal}
\usepackage[all]{xy}
\usepackage{graphicx}
\usepackage{pifont}
\usepackage{comment}
\usepackage{tikz}
\usetikzlibrary{matrix,arrows}

\newcommand\nc{\newcommand}
\nc{\on}{\operatorname}
\nc{\R}{\mathbb R}
\nc{\C}{\mathbb C}
\nc{\Q}{\mathbb Q}
\nc{\Z}{\mathbb Z}
\nc{\N}{\mathbb N}
\nc{\F}{\mathbb F}
\nc{\Hom}{\on{Hom}}
\nc{\wt}{\widetilde}
\nc{\kernel}{\text{ker}}
\nc{\image}{\text{Im}}
\nc{\sls}{\subsetneq ... \subsetneq}
\nc{\ssn}{\subsetneq}
\nc{\bull}{$\bullet \, \,$}
\nc{\ol}{\overline}
\nc{\short}[3]{0 \longrightarrow #1 \longrightarrow #2 \longrightarrow #3 \longrightarrow 0}
\nc{\pd}[2]{\frac{\partial #1}{\partial #2}}
\nc{\one}{\mathbf{1}}
\nc{\rnc}{\renewcommand}
\nc{\e}{\varepsilon}
\nc{\DMO}{\DeclareMathOperator}
\nc{\dd}{\emph{d}}
\nc{\grad}{\nabla}

\rnc{\leq}{\leqslant}
\rnc{\geq}{\geqslant}
\rnc{\int}{\varint}
\rnc{\d}{\text{d}}

\DeclareMathOperator{\E}{\mathbb{E}}

\DMO{\D}{\text{D}}
\DMO{\pv}{p.v.}

\newenvironment{nouppercase}{%
  \renewcommand{\uppercasenonmath}[1]{}}{}
\begin{document}


\title{\large Local Correlation and Gap Statistics under Dyson Brownian Motion for Covariance Matrices}

\author{\large Kevin Yang$^{\dagger}$} \thanks{$\dagger$ Stanford University, Department of Mathematics. Email: kyang95@stanford.edu.}

\begin{nouppercase}
\maketitle
\end{nouppercase}
\begin{center}
\today
\end{center}

\begin{abstract}
In this paper, we study the convergence of local bulk statistics for linearized covariance matrices under Dyson's Brownian motion. We consider deterministic initial data $V$ approximate the Dyson Brownian motion for linearized covariance matrices by the Wigner flow. Using universality results for the Wigner flow, we deduce universality for the linearized covariance matrices. We deduce  bulk universality of averaged bulk correlation functions for both biregular bipartite graphs and honest covariance matrices. We also deduce a weak level repulsion estimate for the Dyson Brownian motion of linearized covariance matrices.
\end{abstract}

\tableofcontents

\newpage
\section{Introduction}
The Wigner-Dyson-Gaudin-Mehta conjecture asserts that the local eigenvalue statistics, such as eigenvalue gaps and other statistics on the gap scale, are universal, depending only on the symmetry class of the matrix ensemble (real symmetric, complex Hermitian, quaternion). Towards this conjecture, Freeman Dyson introduced the Dyson Brownian motion for eigenvalues of a mean-field real Wigner matrix in \cite{D} interpolating between the random matrix ensemble of interest and the GOE, whose eigenvalues exhibit the following density (see \cite{GM}):
\begin{align}
p_G(\lambda_1, \ldots, \lambda_N) \ = \ \frac{1}{Z_G} \prod_{i < j} \ |\lambda_i - \lambda_j|^{\beta} e^{-\frac{N}{4} \sum_{i = 1}^N \ \lambda_i^2}.
\end{align}
Dyson \cite{D2}, Gaudin and Mehta \cite{GM} were able to compute the distribution of gaps for the GOE. Moreover, in the papers \cite{BHKY}, \cite{ESY}, \cite{ESYY}, \cite{EYY}, \cite{HLY}, \cite{LSY}, and \cite{LY}, the Dyson Brownian motion was studied in detail, culminating in proofs of the WDGM conjecture for a variety of random Wigner matrix ensembles. The proof of the WDGM conjecture in these papers followed the robust three-step strategy:
\begin{itemize}
\item Step 1: Derive a local law, showing convergence of the Stieltjes transform of the random matrix ensemble at microscopic scales.
\item Step 2: Show short-time stability for times $t \leq N^{-1 + \e}$ of the eigenvalue statistics under Dyson's Brownian motion (DBM), which stochastically interpolate between the random matrix ensemble of interest and the limiting universal ensemble (e.g. GOE).
\item Step 3: Show a short-time to convergence under the interpolating DBM, i.e. show that after time $t \geq N^{-1 + \delta}$ the eigenvalue statistics of the evolved matrix ensemble agree with those of the limiting ensemble.
\end{itemize}
The third step of this three-step strategy has attracted much attention from the perspective of statistical mechanics due to Dyson's work \cite{D} and \cite{D2}. For Wigner matrices, the papers listed above address the third step for Wigner matrices in a variety of ways, from convexity methods to ideas from nonlinear parabolic equations. In the papers \cite{LSY} and \cite{LY}, the third-step was completed for Wigner matrices for a wide class of initial data, expanding on previous work which required strong independence and moment assumptions. This result from \cite{LSY} and \cite{LY} was an important ingredient in proving bulk universality for sparse and correlated Wigner matrices (see \cite{BHKY} and \cite{HLY}), e.g. those coming from Erdos-Renyi graphs and $d$-regular graphs. 

For covariance matrices, however, the third step has been studied in relatively little detail and generality. The papers \cite{ESY} and \cite{ESYY} complete the third-step of the robust strategy outlined above with strong probabilistic assumptions on the initial data matrix ensemble. These assumptions, for example, fail for matrix ensembles of sparse covariance matrices and correlated covariance matrices. These "pathological" ensembles include sparse data in statistics and biregular bipartite graphs in combinatorics/statistical physics. 

The contribution of this paper is to relax the conditions on the initial data to prove a short time to equilibrium under the Dyson Brownian motion for covariance matrices. The approach we take differs from those of \cite{ESY} and \cite{ESYY}. In particular, given a covariance matrix of the form $X_{\ast} = H^{\ast} H$, we instead look at the augmented linearization given by
\begin{align}
X \ := \ \begin{pmatrix} 0 & H \\ H^{\ast} & 0 \end{pmatrix}.
\end{align}
This perspective in studying the covariance matrix DBM seems to be original to this paper, as far as the author is aware.

One outstanding issue preventing the ideas used in \cite{LY} to prove gap universality and bulk averaged correlation universality for covariance matrices or their linearizations is the lack of a level-repulsion estimate for these two ensembles. The level repulsion estimate for the Wigner flow was a crucial ingredient in \cite{LY}. However, the level-repulsion estimate seems to be currently unavailable for linearized covariance matrices, since these matrices are not mean-field (unlike the Wigner ensemble flow). 

In this paper, we compute the stochastic differential equations for the eigenvalue dynamics under the Dyson Brownian motion on this linearization. With a strong local law and rigidity estimates, we then follow the ideas of \cite{LSY} and construct a short-range cut-off for the eigenvalue dynamics. The key idea here is that the SDEs for short-range eigenvalues identifies with the SDEs for short-range eigenvalues of Wigner matrices. Because the short-range dynamics approximate the honest dynamics (for both linearized covariance matrices and Wigner matrices) to a scale $N^{-1 - \delta}$, it suffices to study local statistics of the cut-off equations. Using the results in \cite{LY} concerning the Wigner flow, we thus deduce universality for the linearized covariance flow. In particular, this paper avoids using the level-repulsion estimate directly for linearized covariance matrices, but rather uses the level-repulsion estimate for Wigner matrices instead. By the comparison with Wigner flow, we also deduce a weak level-repulsion for linearized covariance matrices, but one too weak to use the arguments of \cite{LY}. However, this implies that the universality problem for covariance matrices is more fruitful when looking at the linearization rather than the covariance matrix itself. This idea is also original to this paper as far as the author is aware.
\subsection{Acknowledgements}
The author thanks H.T. Yau, Benjamin Landon, and Patrick Lopatto for answering the author's questions pertaining to random regular graphs. This paper was written while the author was a student at Harvard University.
\subsection{Notation}
We adopt the Landau notation for big-Oh notation, and the notation $a \lesssim b$. For a real symmetric matrix $H$, we let $\sigma(H)$ denote its (real) spectrum.

%
%
%
\section{Matrix dynamics and the free convolution}
We let $\mathscr{M}$ denote the (finite-dimensional) Hilbert space of $M \times N$ matrices with real entries equipped with the following Hilbert-Schmidt inner product for any two matrices $A, B \in \mathscr{M}$:
\begin{align}
\langle A, B \rangle \ = \ \on{Tr}(A^{\ast} B).
\end{align}
Here, we assume, without loss of generality that $\alpha := M/N \geq 1$. Because $\mathscr{M}$ is a finite-dimensional Hilbert space, standard Gaussian theory (see Proposition 2.1 in \cite{Y2}) implies the existence of a basis-independent Gaussian measure and consequently a basis-independent Brownian motion on $\mathscr{M}$. This allows us to define the following matrix-valued SDE on $\mathscr{M}$:
\begin{align}
\d H(t) \ = \ \frac{1}{\sqrt{N}} \d\mathbf{B}(t), \quad H(0) \ = \ H. \label{eq:BMMatrix}
\end{align}
In contrast to \cite{Y2}, this is the SDE of interest for this paper; up to a simple change of variables in time this is equivalent to the matrix-valued Ornstein-Uhlenbeck process studied in \cite{Y2} for short times. We now introduce the space of interest:
\begin{align}
\mathscr{M}_{\ell} \ := \ \left\{ X \ = \ \begin{pmatrix} 0 & H^{\ast} \\ H & 0 \end{pmatrix}, \quad H \in \mathscr{M} \right\}.
\end{align}
As in \cite{Y2}, the space $\mathscr{M}_{\ell}$ inherits a Hilbert space structure and the corresponding Brownian motion from $\mathscr{M}$. The SDE \eqref{eq:BMMatrix} on $\mathscr{M}$ induces the following Brownian motion SDE on $\mathscr{M}_{\ell}$:
\begin{align}
\d X(t) \ = \ \begin{pmatrix} 0 & \d H(t) \\ \d H(t)^{\ast} & 0 \end{pmatrix}, \quad X(0) \ = \ X \ = \ \begin{pmatrix} 0 & H(0) \\ H(0)^{\ast} & 0 \end{pmatrix}. \label{eq:BMMatrixAugment}
\end{align}
We note that the matrix $X(t)$ may be parameterized via the following Gaussian perturbation of the initial data:
\begin{align}
X(t) \ = \ X(0) \ + \ \sqrt{t} \begin{pmatrix} 0 & W \\ W^{\ast} & 0 \end{pmatrix} \ =: \ X \ + \ \sqrt{t} W_{\ell},
\end{align}
where $W$ is a random matrix whose entries are i.i.d. centered Gaussian random variables with variance $1/N$. In particular, borrowing the notation from \cite{Y2}, which we will do in this paper from now on, we have $W_{\ell} = (W, W^{\ast})$.

We now define a notion of regularity for the initial data $X(0)$. First, for notational simplicity and consistency with \cite{LY}, we establish the notation $V = H(0)$.
\begin{assumption} \label{assumption:diagonalV}
The initial data $V$ is $M \times N$ diagonal, i.e. $V$ has the following form:
\begin{align}
V \ = \ \begin{pmatrix} 
V_1 & 0 & 0 & \ldots \\
0 & V_2 & 0 & \ldots \\
\vdots & \vdots & \ddots & \vdots \\
0 & 0 & \ldots & V_N \\
0 & 0 & \ldots & 0 \\
\ldots & \ldots & \ldots & \ldots \\
0 & 0 & \ldots & 0
\end{pmatrix}.
\end{align}
\end{assumption}
The representation $V$ may be attained from a (time-dependent) singular value decomposition. Indeed, the trace inner product on $\mathbf{M}_{M \times N}(\R)$ is invariant under the SVD, and thus so is the Brownian motion $\mathbf{B}(t)$.
\begin{remark}
The main result of this paper will be deterministic in $V$, which allows us to take an SVD of the initial data in the overall goal of universality for linearized covariance matrix ensembles.
\end{remark}

We now introduce the pair of parameters $(g_N, G_N) = (g, G)$ satisfying the following bounds for some fixed $\e_1 > 0$:
\begin{align}
\frac{1}{N} \ \leq \ g \ \leq \ N^{-\e_1}, \quad N^{\e_1} \ell \ \leq \ G \ \leq \ \mathscr{G}_{\e_1}(N),
\end{align}
where we take either $\mathscr{G}_{\e_1}(N) = N^{-\e_1}$ or $\mathscr{G}_{\e_1}(N) \asymp 1$.

For a real number $E_0$ we define the following interval, or energy window,
\begin{align}
\mathscr{I}_{E_0, G} \ = \ (E_0 \pm G).
\end{align}
Lastly, we define the following pseudo-Stieltjes transform for the initial data $V$ as a function of $z = E + i \eta \in \C_+$:
\begin{align}
m_{V}(E + i \eta) \ := \ \frac{1}{N} \sum_{i = 1}^N \ \frac{1}{V_i - E - i \eta}.
\end{align}
%
\begin{definition}
The initial potential $V$ is $(g, G)$-\emph{regular} at $E_0$ if the following pseudo-Stieltjes transform bounds hold uniformly over $z = E + i \eta$ with $E \in \mathscr{I}_{E_0, G}$ and $\eta \in [g, 10]$:
\begin{align}
c_V \ \leq \ \on{Im} m_V(E + i \eta) \ \leq \ C_V
\end{align}
for some $N$-independent constants $C_V, c_V > 0$. Moreover, we require the following spectral bounds:
\begin{itemize}
\item For some $N$-independent constant $B_V > 0$, we have
\begin{align}
\| V \|_{\infty} \ \leq \ N^{B_V},
\end{align}
where the norm $\| \cdot \|_{\infty}$ is the operator or spectral norm.
\item If $M > N$, we require, in addition for some $N$-independent constant $\e > 0$, the following lower bound:
\begin{align}
\inf_{i \in [[1, N]]} \ |V_i| \ &> \ \e \ > \ 0.
\end{align}
\end{itemize}
The lower bound on the singular values in the regime $M > N$ serves to regularize otherwise singular dynamics that we will define shortly. We note that for a large class of random matrix ensembles satisfying a local law, the uniform lower bound on the spectrum of $V$ is also satisfied, and is thus a reasonable constraint to impose.
\end{definition}
We now introduce an a priori delocalization estimate on the eigenvectors of the initial data $X(0) = (V, V^{\ast})$. This will not be necessary as the results in this paper may be obtained without eigenvector delocalization, but it will simplify arguments later in this paper. Moreover, any matrix ensemble satisfying a local law along the diagonal of the Green's function will satisfy the following delocalization estimate. 
\begin{ap}
Suppose $X(t)$ solves the SDE \eqref{eq:BMMatrixAugment} with $H(0) = V$. Then the eigenvectors of $X(t)$ are delocalized with high probability, i.e. for any growth parameter $\xi$ satisfying $\xi \log \xi \gg \log^2 N$ and time $t \geq 0$, we have
\begin{align}
\mathbb{P} \left( \sup_{\lambda(t) \in \sigma(X(t))} \ \sup_{i > M} \ \left| \mathbf{u}_{\lambda(t)}(i) \right| \ \geq \ \frac{\xi}{\sqrt{N}} \right) \ \leq \ e^{-\frac{\xi^2}{2}},
\end{align}
where we let $\mathbf{u}_{\lambda(t)}$ denote the $\ell^2$-normalized eigenvector of $X(t)$ corresponding to the eigenvalue $\lambda(t)$. Moreover, if $M = N$, then the constraint $i > M$ on the index may be removed.
\end{ap}
Lastly, we define the following sets of allowable times. For $\omega, \delta > 0$ to be determined, define
\begin{align}
\mathscr{T}_{\delta, \omega} \ &:= \ \left\{ t: \ \ell N^{\omega} \leq t \leq \ell N^{\omega + \delta} \right\}, \\
\mathscr{T}_{\omega} \ &:= \ \left\{ t: \ \ell N^{\omega} \leq t \leq G N^{-\omega} \right\}.
\end{align}
%
\subsection{Free convolution law}
We recall the density for the linearized Marchenko-Pastur law:
\begin{align}
\varrho(E) \ = \ \begin{cases}
	\frac{\gamma}{(1 + \gamma) \pi |E|} \sqrt{(\lambda_{+} - E^2)(E^2 - \lambda_{-})} & E^2 \in [\lambda_{-}, \lambda_{+}] \\
	0 & E^2 \not\in [\lambda_{-}, \lambda_{+}]
	\end{cases}.
\end{align}
For a more detailed discussion, we refer to Section 2 in \cite{Y1}. We also introduce the Wigner semicircle density, which describes the global statistics of the GOE ensemble:
\begin{align}
\varrho_{\on{sc}}(E) \d E \ = \ \mathbf{1}_{|E| \leq 2} \frac{\sqrt{4 - E^2}}{2 \pi} \d E. 
\end{align}
While the macroscopic eigenvalue statistics of linearized covariance matrices follow the linearized Marchenko-Pastur law given by the density $\varrho$, the Gaussian perturbation follows a slightly perturbed statistics. We follow the ideas of \cite{LY} and \cite{LS} and introduce the following interpolation of the spectra of $X$ and $\sqrt{t} W_{\ell}$ to address the perturbation in eigenvalue statistics.
\begin{definition}
The \emph{free convolution measure} of the Gaussian perturbation $X + \sqrt{t} W_{\ell}$ is the probability measure corresponding to the Stieltjes transform $m_{\on{fc},t}: \C_+ \to \C_+$, given by the unique solution to the following fixed-point equation:
\begin{align}
m_{\on{fc},t}(z) \ = \ \frac{1}{2N} \sum_{i = 1}^N \ \left( \frac{1}{V_i - z - t m_{\on{fc},t}(z)} + \frac{1}{-V_i - z - tm_{\on{fc},t}(z)} \right), \quad z \in \C_+.
\end{align}
\end{definition}
Here, we take for granted existence and uniqueness of solutions to the above fixed-point equation. We also take for granted the following properties of the solution. For a reference on the free convolution, we cite \cite{Bi}.
\begin{itemize}
\item The free convolution measure has a density $\varrho_{\on{fc},t}$ absolutely continuous with respect to Lebesgue measure on $\R$.
\item The density $\varrho_{\on{fc},t}$ is compactly supported and analytic on the interior of its support.
\end{itemize}
To each law $\varrho$, $\varrho_{\on{fc}}$, and $\varrho_{\on{sc}}$, we introduce the classical locations $\gamma_{i}$, $\gamma_{i,t}$, and $\mu_i$, respectively, via the following quantile formulas:
\begin{align}
\frac{i}{N} \ = \ \int_{-\infty}^{\gamma_{i}} \ \varrho(E) \ \d E \ = \ \int_{-\infty}^{\gamma_{i,t}} \ \varrho_{\on{fc},t}(E) \ \d E \ = \ \int_{-\infty}^{\mu_i} \ \varrho_{\on{sc}}(E) \ \d E.
\end{align}
These will be necessary in both stating and deriving the universality of eigenvalue statistics.
\subsection{Main results: correlation and gap universality}
We now introduce the two main results of this paper. The first concerns eigenvalue gap statistics, comparing the gap statistics between adjacent eigenvalues for the perturbed matrix $X(t) = X + \sqrt{t} W_{\ell}$ to the statistics for the GOE ensemble. 
\begin{theorem} \label{theorem:universalgaps}
Suppose $X(t) = X + \sqrt{t} W_{\ell}$, and assume $X = (V, V^{\ast})$, where $V$ is $(g, G)$-regular at $E_0$. Suppose $t \in \mathscr{T}_{\delta, \omega}$ for sufficiently small $\omega, \delta > 0$. Then, for any $O \in C_c^{\infty}(\R^n)$ and index $i$ such that $\gamma_{i,t} \in \mathscr{I}_{E_0, G/2}$, we have, for any indices $i_1, \ldots, i_n \leq N^{c_{\omega}}$,
\begin{align}
\E^{X(t)} &\left[ O \left( N \varrho_{\on{fc},t}(\gamma_{i,t}) (\lambda_{i} - \lambda_{i + i_1} ), \ldots, N \varrho_{\on{fc},t}(\gamma_{i,t}) (\lambda_{i} - \lambda_{i + i_n}) \right) \right] \nonumber \\
&- \ \E^{\on{GOE}} \left[ O \left( N \varrho_{\on{sc}}(\mu_i) (\lambda_{i} - \lambda_{i+i_1}), \ldots, N \varrho_{\on{sc}}(\mu_i)(\lambda_{i} - \lambda_{i + i_n} ) \right) \right]  \ \leq \ N^{-c_{\omega}}
\end{align}
for some universal constant $c_{\omega} > 0$.
\end{theorem}
We note that Theorem \ref{theorem:universalgaps} gives a rough rate of decay explicitly in the gap statistics. We also remark the universality of the constant $c_{\omega}$ follows from the spectral bound on the initial data $V$ and depends on at most the $C^1$-norm of the test function $O$. These dependencies will follow from the proof of Theorem \ref{theorem:universalgaps}.

From Theorem \ref{theorem:universalgaps}, we may obtain the second result of this paper giving bulk universality of averaged correlation functions for perturbed matrices $X(t)$, which will avoid giving an explicit rate of decay and will instead pass to the limit $N \to \infty$. 
\begin{theorem} \label{theorem:universalcorrelation}
Assume the setting of Theorem \ref{theorem:universalgaps}, and for any fixed $c < \omega/2 \wedge \delta/2$, define the parameter $b = N^c/N$. Then for any $O \in C_{c}^{\infty}(\R^n)$, we have for any $E''$ in the interior of the support of $\varrho$:
\begin{align}
\int_{E_0 - b}^{E_0 + b} \ \frac{\dd E'}{2b} \int_{\R^n} \ &  O(\alpha_1, \ldots, \alpha_n) \nonumber \\
&\left[ \frac{1}{(\varrho_{\on{fc},t}(E_0))^n} \varrho_t^{(n)} \left( E' + \frac{\alpha_1}{N \varrho_{\on{fc},t}(E_0)}, \ldots, E' + \frac{\alpha_n}{N \varrho_{\on{fc},t}(E_0)} \right) - \begin{pmatrix} \varrho_{\on{fc},t} \rightarrow \varrho_{\on{sc}} \\ E' \rightarrow E'' \\ \varrho_t^{(n)} \rightarrow \varrho_{\on{GOE}}^{(n)} \end{pmatrix} \right] \ \dd^n \vec{\alpha} \ \rightarrow_{N \to \infty} \ 0,
\end{align}
where the vector notation denotes the same term but with replacements, e.g. replacing $\varrho_{\on{fc},t}$ with $\varrho_{\on{sc}}$.
\end{theorem}
The method for deriving the universality of correlation functions from the gap universality may be found in \cite{ESY} and \cite{LY}, so we omit it.
\subsection{Applications to random matrix ensembles}
We now apply Theorem \ref{theorem:universalcorrelation} to the matrix ensemble of bipartite graphs. By Theorem 2.5 in \cite{Y1}, we deduce that with high probability the adjacency matrix of a bipartite graph is $(g, G)$-regular. To be precise, we let $\Xi$ denote an event holding with high probability on which the adjacency matrix is regular. Because Theorem \ref{theorem:universalcorrelation} is deterministic in the initial data we deduce that on $\Xi$, the averaged local correlation functions of the bipartite graph, after evolving under the DBM for time $t = N^{-1 + \e}$, coincide with those of the GOE. Meanwhile, by the bound $\mathbb{P} \Xi^C \leq N^{-D}$ for all $D > 0$, the contribution from the expectation term in Theorem \ref{theorem:universalcorrelation} on the complement event $\Xi^C$ is negligible in the limit as $N \to \infty$. We thus deduce that, after time $t = N^{-1 + \e}$, the averaged local correlation functions of bipartite graphs and the GOE coincide. Thus by the short-time stability of correlation functions from Theorem 2.11 and 2.12 in \cite{Y2}, this completes the proof of universality for averaged bulk local correlation functions for bipartite graphs.

We conclude this introduction by applying Theorem \ref{theorem:universalgaps} and Theorem \ref{theorem:universalcorrelation} to other covariance matrix ensembles of interest. Applying Theorem \ref{theorem:universalcorrelation} in the following manners still leaves incomplete answers to important questions, but we provide the following outline to give a broad picture of the universality problem for covariance matrices.
\begin{itemize}
\item Concerning the ensemble of sparse covariance matrices as studied in \cite{A}, by Theorem \ref{theorem:universalcorrelation}, to prove universality of local correlation functions it suffices to show short-time stability of eigenvalue statistics. To this end, the arguments in \cite{HLY} applied to linearized covariance matrices should suffice in proving the desired short-time stability.
\item Concerning the strongly regular initial data in \cite{NP2}, by Theorem \ref{theorem:universalgaps}, to show universality of eigenvalue gaps for the corresponding linearization matrices it now suffices to show short-time stability of gap statistics. This seems to require an optimal level repulsion estimate for linearized covariance matrices, which unfortunately seems to be out of reach for now. We remark on this later in this paper.
\end{itemize}
%
%
%
\section{Strong local law and rigidity}
In this section, we derive the necessary local law and rigidity estimates that will help us control eigenvalues in using stochastic analysis to study the DBM equations. We emphasize the ideas and techniques in deriving these estimates are quite distinct from the methods used to analyze the DBM equations, so the reader is welcome to skip the proof of the main results in this section and go straight for the derivation and analysis of the DBM equations in subsequent sections.
\subsection{Main results}
We begin by introducing the following notion of high probability.
\begin{definition} \label{definition:highprobabilitylargedeviationsappendix}
We say an event $\Omega$ holds with $(\xi,\nu)$-\emph{high probability} if
\begin{align}
\mathbb{P} \left( \Omega^C \right) \ \leq \ e^{-\nu \log^{\xi} N}.
\end{align}
We say an event $\Omega_1$ holds with $(\xi,\nu)$-high probability \emph{on $\Omega_2$} if 
\begin{align}
\mathbb{P} \left( \Omega_1^C \cap \Omega_2 \right) \ \leq \ e^{-\nu \log^{\xi} N}.
\end{align}
\end{definition}
Lastly, before we state the local law for $X(t)$, we define the following spectral domains:
\begin{align}
\mathbf{D}_1 \ &= \ \left\{ z = E + i \eta: \ E \in \mathscr{I}_{E_0, qG}, \quad \eta \in [N^{-1} \varphi^L, 10] \right\}, \\
\mathbf{D}_2 \ &= \ \left\{ z = E + i \eta: \ |E| \leq N^{5 B_V}, \quad \eta \in [10, N^{10B_V + 1}] \right\} \\
\mathbf{D}_{L, q} \ &= \ \mathbf{D}_1 \cup \mathbf{D}_2.
\end{align}
We also define the following \emph{partial} Stieltjes transform of $X(t)$:
\begin{align}
m_N(z; t) \ = \ \frac{1}{2N} \sum_{i = 1}^N \ \left( \frac{1}{\lambda_i - z} + \frac{1}{- \lambda_i - z} \right),
\end{align}
where the sum is taken over eigenvalues $\pm \lambda_i$ such that $|\lambda_i|^2 \in \sigma(X(t)^{\ast} X(t))$.
\begin{theorem} \label{theorem:stronglocallawstronglocallawdeformedmatrices}
Suppose $V$ is $(g, G)$-regular at $E_0$. Let
\begin{align}
\xi \ = \ \frac{A_0 + o(1)}{2} \log \log N
\end{align}
and fix $q < 1$ and $L \geq 40 \xi$. For any $z \in \mathbf{D}_{L,q}$, we have
\begin{align}
\sup_{t \in \mathscr{T}_\omega} \ \left| m_N(z; t) - m_{\on{fc}, t}(z) \right| \ \leq \ \frac{\varphi^{C_1 \xi}}{N \eta}
\end{align}
with $(\xi, \nu)$-high probability and $M, N$ sufficiently large. Here, $\nu$ and $C_1$ are constants depending on all data involved in the statement of this result \emph{except} the dimension $M, N$. The notion of $M, N$ sufficiently large is in the sense of depending on all other data involved in the statement of this theorem.
\end{theorem}
From the local law, we deduce an optimal rigidity estimate on the location of eigenvalues. We first establish the following notation: for a constant $0 < q < 1$ and a time $t$ we define
\begin{align}
\mathscr{A}_{q,t} \ = \ \left\{i: \ \gamma_{i,t} \in \mathscr{I}_{E_0, qG} \right\}
\end{align}
for some fixed energy $E_0$. The following rigidity estimate is a consequence of Theorem \ref{theorem:stronglocallawstronglocallawdeformedmatrices} by standard methods in Helffer-Sjostrand functional calculus; for details, we refer to \cite{LY}.
\begin{theorem} \label{theorem:optimalrigiditystronglocallawdeformedmatrices}
For some constants $\nu, c_4 > 0$ depending on the data in Theorem \ref{theorem:stronglocallawstronglocallawdeformedmatrices}, for $i \in \mathscr{A}_{q,t}$, we have
\begin{align}
\sup_{t \in \mathscr{T}_\omega} \ \left| \lambda_{i,t} - \gamma_{i,t} \right| \ \leq \ \frac{\varphi^{c_4 \xi}}{N} \label{eq:rigidity}
\end{align}
with $(\xi,\nu)$-high probability. In particular, if $V$ is $(g, G)$-regular for $G \asymp C$ sufficiently large, and if $\| V \|_{\infty} = O(1)$, then the rigidity estimate \eqref{eq:rigidity} holds for all indices.
\end{theorem}
\subsection{Derivation of the self-consistent equation}
The following matrix identity will be crucial in deriving a fixed-point equation for the Stieltjes transform $m_N(z; t)$. Before we give the statement of the result, we first recall the following notation for any set of indices $\mathbb{T} \subset \{1, \ldots, M+N\}$.
\begin{notation}
We let $X^{(\mathbb{T})}$ denote the matrix obtained from removing the rows and columns indexed by elements of $\mathbb{T}$. We let $G^{(\mathbb{T})}$ denote the Green's function of $X^{(\mathbb{T})}$.
\end{notation}
%
\begin{lemma} \label{lemma:GFmatrixequation}
Retaining the definition of $X(t)$, for any index $i \in [[1, N]]$, we have with $\mathbb{T} = (i, M+i)$, 
\begin{align}
\begin{pmatrix}
G_{ii} & G_{i, M+i} \\
G_{M+i, i} & G_{M+i, M+i}
\end{pmatrix} \ = \ \left( \begin{pmatrix} -z & V_i + \sqrt{t} W_{ii} \\ V_i + \sqrt{t} W_{ii} & -z \end{pmatrix} - t v^\ast G^{(\mathbb{T})} v \right)^{-1} \label{eq:GFmatrixequation}
\end{align}
where
\begin{align}
v^{\ast} \ = \ \begin{pmatrix}
0 & 0 & \ldots & 0 & W_{12} & W_{13} & \ldots & W_{1N} \\
W_{21} & W_{31} & \ldots & W_{M1} & 0 & 0 & \ldots & 0
\end{pmatrix}.
\end{align}
\end{lemma}
Indeed, Lemma \ref{lemma:GFmatrixequation} follows from the Schur Complement Formula, removing all rows and columns from $X(t)$ except those corresponding to the indices $i, M+i$.

To organize the derivation of the self-consistent equation more systematically, we introduce the following pieces of notation.
\begin{notation}
For the index $k \in [[1, N]]$, we let $G_{kk}$ denote the Green's function entry $G_{M+k, M+k}$. For any other index, i.e. those not denoted by $k$, we retain the usual matrix entry notation. For the following calculations, we let $k \leftrightsquigarrow M+1$ correspond to the first index larger than $M$.
\end{notation}
\begin{notation}
We denote the entries of the matrix $v^{\ast} G^{(\mathbb{T})} v$ as in Lemma \ref{lemma:GFmatrixequation} by $M_{ij}$ for $i, j = 1,2$. Thus, we have the following formulas defining $M_{ij}$:
\begin{align}
M_{11} \ &= \ \sum_{k, \ell = 1}^{N-1} \ W_{1, k+1} G^{(\mathbb{T})}_{M- 1 + k, M - 1 + \ell} W_{1, \ell + 1}, \\
M_{12} \ &= \ \sum_{k = 1}^{M-1} \ \sum_{\ell = 1}^{N-1} \ W_{k+1, 1} G_{k, M - 1 + \ell}^{(\mathbb{T})} W_{1, \ell + 1}, \\
M_{21} \ &= \ \sum_{k = 1}^{N-1} \ \sum_{\ell = 1}^{M-1} \ W_{1, k+1} G^{(\mathbb{T})}_{M - 1 + k, \ell} W_{\ell + 1, 1}, \\
M_{22} \ &= \ \sum_{k, \ell = 1}^{M-1} \ W_{k+1, 1} G^{(\mathbb{T})}_{k \ell} W_{\ell+1, 1}.
\end{align}
\end{notation}
By Lemma \ref{lemma:GFmatrixequation}, we deduce
\begin{align}
G_{kk} \ &= \ \frac{-z - tM_{11}}{(-z - t M_{11})(-z - t M_{22}) \ - \ (V_1 + \sqrt{t} W_{11} - t M_{12})(V_1 + \sqrt{t} W_{11} - t M_{21})} \\
&= \ \frac{-z - tM_{11}}{(-z - t M_{11})(-z - t M_{22}) \ - \ \left| V_1 + \sqrt{t} W_{11} - t M_{12}\right|^2}, \label{eq:repGFdiagonal}
\end{align}
where the second equation holds by the Hermitian property of the Green's function. We now expand the above representation of $G_{kk}$. To do so, we first define the following control parameters, most of which will serve as error terms in the desired self-consistent equation for $m_N(z; t)$. 
\begin{align}
g_1^{\pm} \ &= \ \frac{1}{\pm V_1 - z - t m_{N}(z;t)}, \\
\psi^{(\mathbb{T})} \ &= \ m_N(z;t) - m_N^{(\mathbb{T})}(z;t).
\end{align}
With these main terms, we define the following eventual error terms:
\begin{align}
F_1 \ &= \ t^2 \left( \frac{M_{11} - M_{22}}{2} \right)^2, \\
F_2 \ &= \ \frac{M_{11} + M_{22}}{2} - m_N^{(\mathbb{T})}(z; t), \\
F_3^{\pm} \ &= \ g_1^{\pm} \left( \pm \sqrt{t} W_{11} \mp t A_{12} - t F_2 + t \psi^{(\mathbb{T})} \right).
\end{align}
Lastly, we define the following accumulated error terms which will collect all error terms in the self-consistent equation:
\begin{align}
\mathscr{E}_{1} \ &= \ F_3^+ + F_3^- + F_3^+ F_3^- - g_1^+ g_1^- F_1, \\
\mathscr{K}_{1} \ &= \ \frac{-t[(M_{11} - m_N^{(\mathbb{T})}(z;t)] - \psi^{(\mathbb{T})}}{-z - t m_N(z;t)}.
\end{align}
For general indices, the error parameters are defined by changing $1 \to i$ in all of the definitions.

We now state the following result from which we deduce a fixed-point equation for $m_N(z; t)$ up to a controllable error term.
\begin{prop} \label{prop:expansionGFdiagonal}
In the setting above, for indices $k = M+1, \ldots, M+N$, defining $i = k - M$, we have
\begin{align}
G_{kk} \ = \ \frac12 \left( g_i^+ + g_i^- \right) \frac{1 + \mathscr{E}_i}{1 + \mathscr{K}_i}. \label{eq:expansionGFdiagonal}
\end{align}
\end{prop}
\begin{proof}
We focus on the case $k = M+1$ for simplicity; the case for other indices follows from a permutation of the rows and entries of $X(t)$. We first note the following identity controlling the first denominator term in \eqref{eq:repGFdiagonal}:
\begin{align}
(-z - tM_{11})(-z - tM_{22}) \ = \ \left[ - z - \frac{t}{2}(M_{11} + M_{22}) \right]^2 - F_1.
\end{align}
We rewrite the representation \eqref{eq:repGFdiagonal} using the above identity and expand the denominator as follows:
\begin{align}
G_{kk} \ &= \ \frac{-z - t M_{11}}{\left(-z - \frac{t}{2}(M_{11} + M_{22}) \right)^2 - \left(V_1 + \sqrt{t} W_{11} - t M_{12} \right)^2 - F_1} \\
&= \ \frac{1}{\left(-z - tm_N^{(\mathbb{T})} + V_1 + \sqrt{t} W_{11} - t M_{12} - t F_2 \right) \left(-z - t m_N^{(\mathbb{T})} - V_1 - \sqrt{t} W_{11} + tM_{12} - tF_2 \right) - F_1}.
\end{align}
We may rewrite the above equation as follows introducing the term $\psi^{(\mathbb{T})}$ and the error terms $F_3^{\pm}$:
\begin{align}
G_{kk} \ = \ \frac{-z - t M_{11}}{\left(-z - t m_N + V_1 \right)(1 + F_3^+) \left(-z - t m_N - V_1 \right)(1 + F_3^-) - F_1}.
\end{align}
Proposition \ref{prop:expansionGFdiagonal} now follows from a straightforward calculation and the definition of $\mathscr{E}_i, \mathscr{K}_i$.
\end{proof}
If without the error terms $(1 + \mathscr{E}_i)(1 + \mathscr{K}_i)^{-1}$ in the refolded expansion \eqref{eq:expansionGFdiagonal}, we would be able to average over indices $k$ and obtain the same fixed-point equation for the free convolution Stieltjes transform $m_{\on{fc},t}$. To deal with the error terms, we appeal to an a priori estimate which we will derive under suitable conditions. With this control on the error terms, we derive a perturbed form of the fixed-point equation for $m_{\on{fc},t}$. 
\begin{ap}
The following uniform estimate holds with with $(\xi, \nu)$-high probability:
\begin{align}
\frac{1 + \mathscr{E}_i}{1 + \mathscr{K}_i} \ = \ 1 + O \left( \sqrt{\frac{N^{\e}}{N \eta}} \right). \label{eq:errortermestimate}
\end{align}
for $z = E + i \eta \in \mathbf{D}_{L,q}$ and for some constant $C$ and any $\e > 0$.
\end{ap}
With this a priori estimate on the error terms, we deduce the following perturbed self-consistent equation for $m_N$.
\begin{corollary} \label{corollary:sce}
Assuming the upper bound \eqref{eq:errortermestimate}, we have the following equation for all $z \in \mathbf{D}_{L,q}$ and $t \in \mathscr{T}_{\omega}$:
\begin{align}
m_N(z; t) \ = \ \frac{1}{2N} \sum_{i = 1}^N \ \left( \frac{1}{V_i - z - t m_N(z; t)} + \frac{1}{-V_ i - z - t m_N(z; t)} \right) \ + \ O \left( \frac{N^{\e}}{N \eta} \right). \label{eq:sce}
\end{align}
Here, $\e > 0$ is an arbitrarily small, $N$-independent constant.
\end{corollary}
Indeed, upon averaging the identity \eqref{eq:expansionGFdiagonal} over indices $k$ and using the a priori bound \eqref{eq:errortermestimate}, we deduce
\begin{align}
m_N(z; t) \ = \ \frac{1}{2N} \sum_{i = 1}^N \ \left( \frac{1}{V_i - z - t m_N(z; t)} + \frac{1}{-V_i - z - t m_N(z; t)} \right) \left(1 + O \left( \sqrt{\frac{N^{\e}}{N \eta}} \right) \right).
\end{align}
We may thus deduce \eqref{eq:sce} from the following short-time stability estimate particular to the matrix-valued Brownian motion $X(t)$. The proof of this result mirrors the arguments of Lemma 7.5 in \cite{LY} without introducing technical difficulties, so we omit it.
\begin{lemma} \label{lemma:stabilityestimate}
Uniformly over $z \in \mathbf{D}_{L,q}$, we have
\begin{align}
\frac{1}{2N} \sum_{i = 1}^{N} \ \left( \frac{1}{\left| V_i - z - t m_N(z;t) \right|} + \frac{1}{\left| -V_i - z - t m_N(z; t) \right|} \right) \ = \ O(\log N).
\end{align}
\end{lemma}
Using Lemma \ref{lemma:stabilityestimate}, we deduce the following elementary bound on the big-Oh term upon possibly redefining $\e > 0$:
\begin{align}
\left| \frac{1}{2N} \sum_{i = 1}^{N} \ \frac{O \left( \sqrt{\frac{N^{\e}}{N \eta}} \right)}{\pm V_i - z - t m_N(z;t)} \right| \ &\leq \ O \left( \sqrt{\frac{N^\e}{N \eta}} \log N \right) \ \leq \ O \left( \sqrt{\frac{N^\e}{N \eta}} \right).
\end{align}
The signs on the initial data terms $V_i$ indicate a summation over signs as well. This completes the derivation of the self-consistent equation \eqref{eq:sce}.

We now introduce the following a priori estimate for short times $t \in \mathscr{T}_{\omega}$.
\begin{ap}
For any $t \in \mathscr{T}_{\omega}$, we have the following lower bound for any small $\delta > 0$:
\begin{align}
\left| 1 - \frac{t}{2N} \sum_{i = 1}^{N} \ \frac{1}{(\pm V_i - z - tm_N(z;t))(\pm V_i - z - tm_{\on{fc},t}(z))} \right| \ \geq \ N^{-\delta}. \label{eq:secondstabilityestimate}
\end{align}
Here, the signs $\pm$ are chosen to be the same; in particular, only two possible pairs of signs are allowed.
\end{ap}
We now discuss in more detail how to derive the a priori bounds \eqref{eq:errortermestimate} and \eqref{eq:secondstabilityestimate}. To this end we define the difference term of primary interest:
\begin{align}
\Lambda(z) \ := \ \left| m_N(z; t) - m_{\on{fc},t}(z) \right|.
\end{align}
The following result will also allow us to begin our bootstrapping scheme at the scale $\eta = 1$.
\begin{prop} \label{prop:deriveapbounds}
For any fixed $z \in \mathbf{D}_{L,q}$, suppose either $\eta \geq 1$ or the following upper bound holds:
\begin{align}
\Lambda(z) \ = \ O \left( \sqrt{\frac{N^{\e}}{N \eta}} \right).
\end{align}
Then with $(\xi, \nu)$-high probability, the bound \eqref{eq:errortermestimate} holds uniformly over all indices $k$, and the bound \eqref{eq:secondstabilityestimate} also holds.
\end{prop}
\begin{proof}
For the proof of deducing \eqref{eq:errortermestimate}, we refer to the proof of Lemma 7.9 in \cite{LY} which again mirrors the proof of \eqref{eq:errortermestimate}. To derive the bound \eqref{eq:secondstabilityestimate}, we first note it suffices to provide a $O(N^{-\delta})$ upper bound for the following term:
\begin{align}
\frac{t}{2N} \sum_{i = 1}^N \ \frac{1}{(\pm V_i - z - tm_N(z;t))(\pm V_i - z - t m_{\on{fc},t}(z))}.
\end{align}
To derive this upper bound, we begin by rewriting each term in the summation as follows:
\begin{align}
\frac{1}{(V_i - z - tm_N(z;t))(V_i - z - tm_{\on{fc},t}(z))} \ &= \ \frac{1}{(V_i - z - tm_{\on{fc},t}(z))^2} \nonumber \\
&\quad \ + \ \frac{t(m_N(z;t) - m_{\on{fc},t}(z))}{(V_i - z - tm_N(z;t))(V_i - z - tm_{\on{fc},t}(z))^2}.
\end{align}
We handle the first term on the RHS as follows. We cite (7.10) in Lemma 7.2 and (7.26) in Lemma 7.3 in \cite{LY} to obtain the following $O(1)$-bound:
\begin{align}
1 - \frac{t}{2N} \sum_{i = 1}^{N} \ \frac{1}{(\pm V_i - z - tm_{\on{fc},t}(z))^2} \ = \ O(1) \label{eq:O(1)bound}
\end{align}
It now remains to bound the following term:
\begin{align}
\frac{t^2(m_N(z;t) - m_{\on{fc},t}(z))}{2N} \sum_{i = 1}^{N} \ \frac{1}{(\pm V_i - z - t m_N(z;t))(\pm V_i - z - t m_{\on{fc},t}(z))^2}.
\end{align}
This follows from the Schwarz inequality and the following estimate uniform all indices $i \in [[1, N]]$:
\begin{align}
t \ \lesssim \ |V_i - z - Tm_{\on{fc},t}(z)| \ \lesssim \ 1,
\end{align}
which, in turn, follow from the estimate $\on{Im} m_{\on{fc},t}(z) \asymp 1$. These estimates hold uniformly over $z \in \mathbf{D}_{L,q}$ and $t \in \mathscr{T}_{\omega}$, which completes the proof of Proposition \ref{prop:deriveapbounds}.
\end{proof}
\subsection{Analysis of the self-consistent equation}
We begin our analysis of the self-consistent equation by establishing the following identity which will be the backbone of our bootstrapping method to obtain local laws at smaller scales $\eta$.
\begin{lemma} \label{lemma:backbone}
Suppose $z \in \mathbf{D}_{L,q}$ and $t \in \mathscr{T}_{\omega}$. The following estimates hold with $(\xi, \nu)$-high probability. Assuming the a priori bound \eqref{eq:errortermestimate}, we have
\begin{align}
\left( m_N(z;t) - m_{\on{fc},t}(z) \right) \left(1 - \frac{t}{2N} \sum_{i = 1}^{N} \ \frac{1}{(\pm V_i - z - t m_N(z;t))(V_i - z - tm_{\on{fc},t}(z))} \right) \ = \ O \left( \sqrt{\frac{N^{\e}}{N \eta}} \right). \label{eq:stabilityequation}
\end{align}
In particular, assuming the bound \eqref{eq:secondstabilityestimate} in addition, we have, for a possibly different $\e > 0$,
\begin{align}
\Lambda(z) \ := \ O \left( \sqrt{\frac{N^{\e}}{N \eta}} \right), \label{eq:desiredlocallaw}
\end{align}
where the implied constant depends only on the equation \eqref{eq:stabilityequation}.
\end{lemma}
\begin{remark}
The estimate \eqref{eq:desiredlocallaw} is known as a \emph{weak local law} in contrast to the strong local law in Theorem \ref{theorem:stronglocallawstronglocallawdeformedmatrices}. 
\end{remark}
\begin{proof}
We note it suffices to prove the first estimate, as the second estimate follows from the first and the lower bound in \eqref{eq:secondstabilityestimate}. We take the difference between the fixed-point equation for $m_{\on{fc},t}$ and \eqref{eq:sce} to obtain the following perturbed equation:
\begin{align}
m_N(z;t) - m_{\on{fc},t}(z) \ = \ \frac{1}{2N} \sum_{i = 1}^{N} \ \left( \frac{1}{\pm V_i - z - tm_N(z;t)} - \frac{1}{\pm V_i - z - t m_{\on{fc},t}(z)} \right) \ + \ O \left( \sqrt{\frac{N^{\e}}{N \eta}} \right). 
\end{align}
Here, we appealed to the a priori bound \eqref{eq:errortermestimate} to derive the above equation. We now rewrite each term in the summation on the RHS as follows:
\begin{align}
\frac{1}{\pm V_i - z - t m_N(z;t)} - \frac{1}{\pm V_i - z - t m_{\on{fc},t}(z)} \ &= \ \frac{t \left( m_N(z;t) - m_{\on{fc}, t}(z) \right)}{(\pm V_i - z - t m_N(z;t))(\pm V_i - z - t m_{\on{fc},t}(z))}. 
\end{align}
Having rewritten the summation term, we subtract it from both sides and obtain a $O \left(\sqrt{\frac{N^{\e}}{N\eta}} \right)$ bound on the following term:
\begin{align}
m_N(z; t) - m_{\on{fc},t}(z) - &\frac{t(m_N(z;t) - m_{\on{fc},t}(z))}{2N} \sum_{i = 1}^N \ \frac{1}{(\pm V_i - z - tm_N(z;t))(\pm V_i - z - tm_{\on{fc},t}(z))} \\
&= \ \left( m_N(z;t) - m_{\on{fc},t}(z) \right) \left( 1 - \frac{t}{2N} \sum_{i = 1}^N \ \frac{1}{(\pm V_i - z - tm_N(z;t))(\pm V_i - z - tm_{\on{fc},t}(z))} \right).
\end{align}
This completes the proof of Lemma \ref{lemma:backbone}.
\end{proof}
To improve the weak local law from the scale $\eta = 1$ to the optimal scale $\eta \gg N^{-1}$, we appeal to the following bootstrapping scheme. 
\begin{prop} \label{prop:weaklocallaw}
Assume the setting of Theorem \ref{theorem:stronglocallawstronglocallawdeformedmatrices}. Then for any $z \in \mathbf{D}_{L,q}$ and any $\e > 0$, we have
\begin{align}
\Lambda(z) \ = \ O \left( \sqrt{\frac{N^{\e}}{N \eta}} \right).
\end{align}
\end{prop}
\begin{proof}
For a fixed $z \in \mathbf{D}_{L,q}$, we define $z_k = E + i \eta_k$. Here, we have defined the parameters as follows:
\begin{align}
E \ = \ \on{Re}(z), \quad \eta_k = 1 - kN^{-4}.
\end{align}
Lastly, we let $K$ denote the maximal integer such that $z_K \in \mathbf{D}_{L,q}$; in particular we have the bound $K = N^{O(1)}$. By $N^2$-Lipschitz continuity of the Stieltjes transforms, and thus of $\Lambda$, to derive the weak local law at $z$, it suffices to derive the weak local law at $z_k$ for all $k = 1, \ldots, K$. We do so inductively in $k$. 

We note the weak local law at $z_0$ holds with $(\xi, \nu)$-high probability by Proposition \ref{prop:deriveapbounds}. This allows us to obtain the following estimate for $z_1$:
\begin{align}
\Lambda(z_1) \ \leq \ \left| \Lambda(z_1) - \Lambda(z) \right| + \Lambda(z)  \ &\leq \ N^2 |z_1 - z| + O \left( \sqrt{\frac{N^{\e}}{N \eta}} \right) \\
&\leq \ N^{-4} + O \left( \sqrt{\frac{N^{\e}}{N \eta}} \right).
\end{align}
By Proposition \ref{prop:deriveapbounds} and Lemma \ref{lemma:backbone}, this implies 
\begin{align}
\Lambda(z_1) \ = \ O \left( \sqrt{\frac{N^{\e}}{N \eta_1}} \right).
\end{align}
Continuing inductively, we see for any $k \in [[1, K]]$ that the weak local law holds:
\begin{align}
\Lambda(z_k) \ = \ O \left( \sqrt{\frac{N^{\e}}{N \eta}} \right),
\end{align}
where the implied constant is independent of the index $k$. This completes the proof.
\end{proof}
With a similar bootstrapping method, we may also obtain an estimate of similar type for the off-diagonal entries of the Green's function $G$. To state this estimate, we first define the following pseudo-maximal functions for $z \in \mathbf{D}_{L,q}$:
\begin{align}
\Lambda_{o}(z) \ = \ \max_{i \neq j > M} \ |G_{ij}(z)|.
\end{align}
To control this term, we appeal to the following result from \cite{LS}, which allows us to control $\Lambda_o$ in terms of the weak local law. 
\begin{lemma} \label{lemma:offdiagonallocallaw}
Assume the following a priori estimate:
\begin{align}
\Lambda_o(z) + \Lambda(z) \ \leq \ \varphi^{-2 \xi}
\end{align}
where we retain the definition of the control parameters $\varphi, \xi$ from the introduction of the underlying model. Then we have
\begin{align} 
\sup_{z \in \mathbf{D}_{L,q}} \ \Lambda_o(z) \ = \ O \left( \varphi^{\xi} \sqrt{\frac{\Lambda(z) + O(1)}{N \eta}} \right)
\end{align}
with $(\xi, \nu)$-high probability, where the implied constant is independent of $z$. 
\end{lemma}
The result in Lemma \ref{lemma:offdiagonallocallaw} is a parallel result to Lemma 3.8 in \cite{LS}. The method of proof is analogous between the two matrix ensembles, so we refer to the details of Lemma 3.8 in \cite{LS} for the proof of Lemma \ref{lemma:offdiagonallocallaw}. We thus obtain the following consequence.
\begin{corollary}
In the setting of Theorem \ref{theorem:stronglocallawstronglocallawdeformedmatrices}, we have the following estimate uniformly over $z \in \mathbf{D}_{L,q}$ with $(\xi, \nu)$-high probability:
\begin{align} 
\sup_{z \in \mathbf{D}_{L,q}} \ \Lambda_o(z) \ = \ O \left( \varphi^{\xi} \sqrt{\frac{\Lambda(z) + O(1)}{N \eta}} \right).
\end{align}
\end{corollary}
The last weak result we discuss here is the following estimates on diagonal entries of the Green's function. In this spirit, we similarly define a pseudo-maximal function controlling the diagonal entries as follows:
\begin{align}
\Lambda_d(z) \ = \ \max_{M + 1 \leq k \leq M+N} \ |G_{kk}(z)|.
\end{align}
On one hand, by eigenvector delocalization and the weak local law, we may bound the diagonal entries by the Stieltjes transform up to a factor of $N^{\e}$ for any small $\e > 0$, and we may compare the Stieltjes transform of $G$ to the Stieltjes transform of the free convolution by the weak local law. Ultimately, this argument provides the following bound for any small $\e > 0$:
\begin{align}
\Lambda_d(z) \ = \ O(N^{\e}).
\end{align}
%
\begin{remark}
We briefly remark that a similar bootstrapping scheme allows for a sharper estimate on $\Lambda_d$ comparing the diagonal entries to the free convolution $m_{\on{fc}}$. We omit the details and refer the reader to \cite{LS}.
\end{remark}
\subsection{Strong local law}
We now focus on improving the weak local law to the strong local law and prove Theorem \ref{theorem:stronglocallawstronglocallawdeformedmatrices}. We begin by introducing the following notation.
\begin{notation}
For a set of indices $\mathbb{T}$, we define the following conditional expectation and fluctuation of a random variable $X$ defined as a function of the matrix entries of $X(t)$:
\begin{align}
\E_{\mathbb{T}} X \ := \ \E \left( X | X(t)^{\mathbb{T}} \right), \quad Q_{\mathbb{T}} X \ := \ X - \E_{\mathbb{T}} X.
\end{align}
\end{notation}
With this notation, we now introduce fluctuation terms that will be important to control in proving Theorem \ref{theorem:stronglocallawstronglocallawdeformedmatrices}; we give them as follows:
\begin{align}
Z_{11}^{(1)} \ &= \ Q_{(1)} \left( \frac{G_{M+1, M+1}}{G_{11} G_{M+1,M+1} - G_{1,M+1} G_{M+1,1}} \right), \\
Z_{12}^{(1)} \ &= \ Q_{(1)} \left( \frac{G_{M+1,1}}{G_{11} G_{M+1, M+1} - G_{1, M+1} G_{M+1, 1}} \right).
\end{align}
We may define the more general fluctuation term replacing the labels $1$ on the RHS above with any index $i \in [[1, M]]$. Secondly, we define another pair of fluctuation terms which will not be immediately important, but we will use them later in deriving the strong local law:
\begin{align}
Z_{21}^{(1)} \ &= \ Q_{(1)} \left( \frac{G_{1,M+1}}{G_{11} G_{M+1, M+1} - G_{1,M+1} G_{M+1,1}} \right), \\
Z_{22}^{(1)} \ &= \ Q_{(1)} \left( \frac{G_{1}}{G_{11} G_{M+1, M+1} - G_{1, M+1} G_{M+1, 1}} \right).
\end{align}
Again, we define $Z_{k\ell}^{(i)}$ similarly by replacing the $1$-indices on the RHS with $i \in [[1, M]]$. 

We now state the fluctuation averaging lemma, which is an important ingredient towards upgrading a weak local law to a strong local law (see \cite{LY} and \cite{LS} for more details).
\begin{prop} \label{prop:fluctuationestimatestronglocallawdeformedmatrices}
Suppose $a_k \asymp 1$ are fixed constants, and assume the following a priori bounds for any fixed $\e > 0$:
\begin{align}
\sup_{\eta \geq N^{-1 + \e}} \ \sup_{i \neq j > M} \ |G_{ij}(E + i \eta)| \ \leq \ O \left( \sqrt{\frac{N^{\e}}{N \eta}} \right), \quad \ \sup_{\eta \geq N^{-1 + \e}} \ \sup_{M + 1 \leq k \leq M+N} \ \left| G_{kk}(E + i \eta) \right| \ = \ O(N^\e). \label{eq:fluctuationapbounds}
\end{align}
Then, we have, for $z = E + i \eta$ with $\eta \geq N^{-1+\e}$,
\begin{align}
\frac{1}{N} \sum_{i = 1}^{N} \ a_k Z_{11}^{(k)} \ &= \ O \left( \frac{N^{\e}}{N \eta} \right), \\
\frac{1}{N} \sum_{i = 1}^N \ a_k Z_{12}^{(k)} \ &= \ O \left( \frac{N^e}{N \eta} \right)
\end{align}
with $(\xi, \nu)$-high probability, where $k = M + i$.
\end{prop}
\begin{proof}
We give an outline of the proof; for details, we refer to Lemma 7.15 in \cite{LY}. In particular, Proposition \ref{prop:fluctuationestimatestronglocallawdeformedmatrices} follows an application of the Chebyshev inequality combined with the following moment bounds:
\begin{align}
\E \left( \frac{1}{N^{2p}} \sum_{i_1, \ldots, i_{2p} = 1}^{N} \ \prod_{j = 1}^{2p} \ a_{k_j} Z_{11}^{(i_j)} \right) \ = \ O \left( \frac{N^{\e}}{N \eta} \right),
\end{align}
and the analogous estimate replacing $Z_{11}$ with $Z_{12}$. This moment bound is quite technical and requires a combinatorial analysis of words on indices. For this reason we omit the proof, but we encourage the reader to look at \cite{LY} for details.
\end{proof}
We are now in a position to derive the strong local law. We first define the following error terms that will arise naturally from studying the matrix equation \eqref{eq:GFmatrixequation}:
\begin{align}
R_{11}^{(1)} \ &= \ \frac{1}{N} G_{M+1, M+1} + \frac{1}{N} \sum_{i = 1}^N \ G_{M+1, k} G_{k, M+1}, \\
R_{22}^{(1)} \ &= \ \frac{1}{N} G_{11} + \frac{1}{N} \sum_{i = 1}^N \ G_{1, k} G_{k,1}.
\end{align}
We similarly define $R_{jj}^{(\ell)}$ upon replacing $\ell = 1$ on the RHS above with any appropriate index $\ell \in [[1, N]]$. We now record the following consequence of the fluctuation averaging lemma:
\begin{align}
\sup_{z \in \mathbf{D}_{L,q}} \ \sup_{i \in [[1, M]]} \ \sup_{j = 1, 2} \ |R_{jj}^{(i)}| \ \leq \ \frac{N^\e}{N \eta}.
\end{align}
Rewriting the matrix equation \eqref{eq:GFmatrixequation} and extracting these error terms, along with the fluctuations $Z$, we obtain
\begin{align}
\begin{pmatrix}
G_{ii} & G_{i, M+i} \\
G_{M+i, i} & G_{M+i, M+i}
\end{pmatrix}^{-1} \ &= \ \begin{pmatrix}
-z - tm_{\on{fc},t}(z) & V_i \\
V_i & - z - tm_{\on{fc},t}(z)
\end{pmatrix} \nonumber \\
&\quad \ + \  \begin{pmatrix}
t \left( m_N(z;t) - m_{\on{fc},t}(z) \right) & 0 \\
0 & t \left( m_N(z;t) - m_{\on{fc},t}(z) \right)
\end{pmatrix} \nonumber \\
&\quad \ + \ t \mathbf{Z}^{(i)} + t \mathbf{R}^{(i)}. \label{eq:rewrittenGFmatrixequation}
\end{align}
Here, the matrices $\mathbf{Z}$ and $\mathbf{R}$ are defined by the terms $Z^{(i)}_{k\ell}$ and $R^{(i)}_{k \ell}$. For convenience, we will introduce the following notation for the matrices on the RHS above:
\begin{align}
\mathscr{Y}_{V_i} \ &= \ \begin{pmatrix} -z - tm_{\on{fc},t}(z) & V_i \\ V_i & -z - tm_{\on{fc},t}(z) \end{pmatrix}, \\
\mathscr{Y}_{d} \ &= \ \begin{pmatrix} t \left( m_N(z) - m_{\on{fc},t}(z) \right) & 0 \\ 0 & t \left( m_N(z) - m_{\on{fc},t}(z) \right) \end{pmatrix}.
\end{align}
By the resolvent identity, we have the following uniform bound for any small $\e > 0$:
\begin{align}
\sup_{i \in [[1,N]]} \ \left\| \frac{1}{N} \sum_{i = 1}^{N} \left[ \left( \mathscr{Y}_{V_i} - \mathscr{Y}_{d} + t \mathbf{Z}^{(i)} + t \mathbf{R}^{(i)} \right)^{-1} \ - \ \left( \mathscr{Y}_{V_i} - \mathscr{Y}_{d} \right)^{-1} \right] \right\|_{\infty} \ \leq \ C \frac{N^{\e}}{N \eta}.
\end{align}
This is a consequence of a straightforward application of the diagonal estimates on $\mathbf{R}^{(i)}$ we obtained as well as the fluctuation averaging lemma, both coupled with the a priori bounds in the statement of Proposition \ref{prop:fluctuationestimatestronglocallawdeformedmatrices}. This allows us to Taylor expand resolvents of matrices in \eqref{eq:rewrittenGFmatrixequation} to obtain the following main term-error term decomposition of the LHS of \eqref{eq:rewrittenGFmatrixequation}:
\begin{align}
\frac{1}{N} \sum_{i = 1}^{N} \ \begin{pmatrix}
G_{ii} & G_{ik} \\
G_{ki} & G_{kk}
\end{pmatrix} \ &= \ \frac{1}{N} \sum_{i = 1}^{N} \ \left( \mathscr{Y}_{V_i} - \mathscr{Y}_{d} \right)^{-1} \ + \ O \left( \frac{N^{\e}}{N \eta} \right) \nonumber \\
&= \ \frac{1}{N} \sum_{i = 1}^{N} \ \mathscr{Y}_{V_i}^{-1} - \frac{1}{N} \sum_{i = 1}^{N} \ \mathscr{Y}_{V_i}^{-1} \mathscr{Y}_{d} \mathscr{Y}_{V_i}^{-1} \ + \ O \left( \frac{N^{\e}}{N\eta} \right) \label{eq:lastGFmatrixequation}
\end{align}
where the big-Oh terms denote matrices with the corresponding operator bound. The second line follows from the geometric series expansion of the inverse of a matrix noting the $o(1)$ estimates on $\Lambda, \Lambda_o$ provided by the weak local law.
\begin{remark}
We note that if an inverse written above does not exist, a perturbation $\e$ in the entries can be introduced and then taken to $\e \to 0$ in the end of this calculation. We do not fully illustrate this in detail for convenience and clarity of the presentation of the calculation. 
\end{remark}
We now take a partial trace of both sides of the last matrix equation \eqref{eq:lastGFmatrixequation}. In particular, we sum over diagonal entries with index larger than $M$. This gives us the following equation:
\begin{align}
\frac{1}{N} \sum_{i = 1}^{N} \ G_{kk} \ &= \ \frac{1}{N} \sum_{i = 1}^{N} \ \frac{-z - tm_{\on{fc},t}(z)}{(-z - tm_{\on{fc},t}(z))^2 - V_i^2} \nonumber \\
&\quad \quad \ - \ \left(m_N(z;t) - m_{\on{fc},t}(z) \right) \frac{t}{N} \sum_{i = 1}^{N} \ \on{Tr} \mathscr{Y}_{V_i}^{-2} \nonumber \\
&\quad \quad \ + \ O \left( \frac{N^{\e}}{N \eta} \right). \label{eq:traceequationstronglocallawdeformedmatrices}
\end{align}
We now study the trace equation \eqref{eq:traceequationstronglocallawdeformedmatrices}, noting the term on the LHS is exactly $m_N$, and the first term on the RHS is exactly $m_{\on{fc},t}$. A straightforward rearrangement of the above trace equation now implies the following estimate:
\begin{align}
(m_N(z;t) - m_{\on{fc},t}(z)) \left( 1 + \mathscr{E}_t \right) \ = \ O \left( \frac{N^{\e}}{N \eta} \right),
\end{align}
where we defined
\begin{align}
\mathscr{E}_t \ = \ \frac{t}{N} \sum_{i = 1}^{N} \ \on{Tr} \mathscr{Y}_{V_i}^{-2}.
\end{align}
Thus, to derive the strong local law for a fixed point $z \in \mathbf{D}_{L,q}$, it suffices to derive a constant \emph{lower} bound on $1 + \mathscr{E}_t$. To this end we appeal to (7.10) in Lemma 7.2 and (7.26) in Lemma 7.3, both in \cite{LY}. This completes the proof of the strong local law for fixed times. To extend this estimate to all times on a high probability event, we appeal to a standard stochastic continuity argument outlined in detail in the appendix of this paper.
%
%
%
\section{Dyson Brownian motion: eigenvalue SDEs}
\subsection{Statement of the DBM equations}
As in the breakthrough paper \cite{D}, our goal is now to compute the associated eigenvalue dynamics of the SDE \eqref{eq:BMMatrixAugment}. To derive these dynamics systematically, we introduce the following spectral set which captures the trivial eigenvalues of the time-parameterized family of matrices $X(t)$:
\begin{align}
\zeta_t(X) \ := \ \left\{ \lambda(t) \in \sigma(X(t)): \ \lambda^2(t) \not\in \sigma(H(t)^{\ast} H(t)) \right\}.
\end{align}
We now describe the spectral set $\zeta_t(X)$ in words. If $M \geq N$, for any eigenvalue $\lambda^2 \in \sigma(H(t)^{\ast} H(t))$, we obtain a pair of eigenvalues $\pm \lambda \in \sigma(X(t))$ in the spectrum of the linearization $X(t)$. The trivial eigenvalues consist of those eigenvalues $\lambda \in \sigma(X(t))$ not obtained from this spectral procedure. This is the set of eigenvalues $\zeta_t(X)$. We thus deduce any $\lambda(t) \in \zeta_t(X)$ is equal to 0, and $\zeta_t(X)$ is empty exactly when $M = N$.
\begin{remark}
We briefly remark here why we define the set $\zeta_t(X)$, as the eigenvalues in this set are trivially 0. In particular, the corresponding eigenvalue dynamics of the matrix-valued SDE \eqref{eq:BMMatrixAugment} will involve terms depending on the repulsion between eigenvalues, including repulsion between eigenvalues both contained in and not contained in the set $\zeta_t(X)$. Thus, we will want to interpret the trivial eigenvalues as honest eigenvalues and not fixed real numbers.
\end{remark}
We now introduce an important assumption on which regularization of the eigenvalue dynamics heavily depends. The assumption takes form of an a priori estimate on the repulsion of nontrivial eigenvalues from 0 and resembles the notion of regularity for our initial potential $V$.
\begin{assumption} \label{assumption:repulsionfrom0DBM}
If $M > N$, then for any fixed time scale $T \in \mathscr{T}_{\delta, \omega}$ and $t \in [0,T]$
\begin{align}
\inf_{\lambda_\alpha \not\in \zeta(X)} \ |\lambda_\alpha| \ > \ \e \ > \ 0,
\end{align}
where we assume $\e > 0$ is independent of $M, N$.
\end{assumption}
We note that Assumption \ref{assumption:repulsionfrom0DBM} with time scale $T$ follows with high probability if the following hold:
\begin{itemize}
\item Assumption \ref{assumption:repulsionfrom0DBM} holds with time scale $T = 0$.
\item The time scale $T$ satisfies $T = o_{N \to 1}(0)$. 
\end{itemize}
This is an immediate consequence of the perturbation inequality for any real symmetric matrices $A, B$:
\begin{align}
\sup_{\lambda \in \sigma(A-B)} \ \left| \lambda \right| \ \leq \ \|A - B \|_{\infty} \label{eq:spectralnormbound}
\end{align}
where the norm on the RHS is the operator norm. Indeed, we parameterize solutions to \eqref{eq:BMMatrixAugment} as follows for time $t \geq 0$:
\begin{align}
X(t) \ = \ \begin{pmatrix} 0 & H + \sqrt{t} W \\ H^{\ast} + \sqrt{t} W^{\ast} & 0 \end{pmatrix},
\end{align}
where $W \in \mathscr{M}$ is sampled from the standard Gaussian measure on $\mathscr{M}$. Letting $A = X$ and $B = X(t)$, coupled with the high-probability estimate on the operator norm $\| W \|_{\infty}$ from the high-probability estimate on centered Gaussian random variables, this gives sufficiency of the conditions on the initial data and the time scale.

We may now introduce the corresponding eigenvalue dynamics. To this end we define the following simplex:
\begin{align}
\Delta_{M+N} \ := \ \left\{ (x_1, \ldots, x_{M+N}): \ x_1 < x_2 < \ldots < x_{M+N} \right\}.
\end{align}
The eigenvalue dynamics will live on the simplex $\Delta_{M+N}$, with the following deterministic equation for $\lambda \in \zeta_t(X)$:
\begin{align}
\d \lambda(t) \ = \ 0, \quad \lambda(0) \ = \ 0.
\end{align}
The SDEs for nontrivial eigenvalues are given as the main result in the following theorem.
\begin{theorem} \label{theorem:DBMDBM}
Suppose $X(t)$ solves the SDE in \eqref{eq:BMMatrixAugment}. Let $\{ \lambda_\alpha(t) \}_{\alpha}$ denote the eigenvalues of $X(t)$ realized as correlated paths on the simplex $\Delta_{M+N}$. Moreover, we assume $N \gg 1$ is sufficiently large so that the following equations are nonsingular.
\begin{enumerate}
\item If $M = N$, then the eigenvalues $\{ \lambda_\alpha(t) \}_\alpha$ solve the following system of SDEs, known as \emph{Dyson's Brownian Motion}:
\begin{align}
\dd\lambda_\alpha(t) \ = \ \frac{1}{\sqrt{N}} \dd B_\alpha(t) \ + \ \frac{1}{2N} \sum_{\beta \neq \pm \alpha} \ \frac{1}{\lambda_\alpha - \lambda_\beta} \ \dd t,
\end{align}
where $\{ B_\alpha(t) \}_\alpha$ denote independent standard one-dimensional Brownian motions with the constraint $B_\alpha(t) = - B_{-\alpha}(t)$. Here, $B_{-\alpha}(t)$ denotes the Brownian motion driving the process defining $-\lambda_\alpha(t)$.
\item Conditioning on Assumption \ref{assumption:repulsionfrom0DBM}, if $M > N$, then the eigenvalues $\{ \lambda_\alpha(t) \}_{\alpha \not\in \zeta_t(X)}$ solve the following system of SDEs:
\begin{align}
\dd \lambda_\alpha(t) \ &= \ \frac{1}{\sqrt{N}} \dd B_\alpha(t) \ + \ \frac{1}{2N} \left( \sum_{\beta \not\in \zeta_t(X)}^{(\pm \alpha)} \ \frac{1}{\lambda_\alpha - \lambda_\beta} \ + \ \frac{M-N}{\lambda_\alpha} \right) \ \dd t \\
&= \ \frac{1}{\sqrt{N}} \dd B_\alpha(t) \ + \ \frac{1}{2N} \left( \sum_{\beta \not\in \zeta_t(X)}^{(\pm \alpha)} \ \frac{1}{\lambda_\alpha - \lambda_\beta} \ + \ \sum_{\gamma \in \zeta_t(X)} \ \frac{1}{\lambda_\alpha - \lambda_\gamma} \right) \ \dd t
\end{align}
where $(\pm \alpha)$ over the summation indicates an omission of terms $\beta = \pm \alpha$. Again, the $\{ B_\alpha(t) \}_\alpha$ are independent standard one-dimensional Brownian motions except for the constraint $B_\alpha(t) = -B_{-\alpha}(t)$, where we retain the notation for $B_{-\alpha}(t)$.
\end{enumerate}
\end{theorem}
\begin{remark}
The regime $M = N$ will be referred to as \emph{square DBM}, and the regime $M > N$ as \emph{rectangular DBM}.
\end{remark}
\subsection*{The Ornstein-Uhlenbeck Variant}
We now briefly discuss a slight variant of the matrix dynamics \eqref{eq:BMMatrix} and \eqref{eq:BMMatrixAugment}. In particular, instead of a classical Brownian motion, we consider the following matrix-valued \emph{Ornstein-Uhlenbeck} (OU) equation:
\begin{align}
\d H(t) \ = \ \frac{1}{\sqrt{N}} \d \mathbf{B}(t) \ - \ \frac12 H(t) \d t, \quad H(0) \ = \ H. \label{eq:OUMatrix}
\end{align}
Similarly, we define the Ornstein-Uhlenbeck equation on the space $\mathscr{M}_{\ell}$ as follows:
\begin{align}
\d X(t) \ = \ \begin{pmatrix} 0 & \d H(t) \\ \d H(t)^{\ast} & 0 \end{pmatrix}, \quad X(0) \ = \ X \ = \ \begin{pmatrix} 0 & H \\ H^{\ast} & 0 \end{pmatrix}. \label{eq:OUMatrixAugment}
\end{align}
In particular, the equation \eqref{eq:OUMatrixAugment} is the matrix-valued SDE studied in detail in \cite{Y2}. On the simplex $\Delta_{M+N}$, we may derive a similar system of SDEs driving the eigenvalue dynamics corresponding to \eqref{eq:OUMatrixAugment}. Before we state this result, we briefly remark on the qualitative similarities and differences when adding the drift term in \eqref{eq:OUMatrix}. 

As for the corresponding eigenvalue SDEs for the equation \eqref{eq:OUMatrixAugment}, by the Ito formula one should expect only a change in the drift term. This is, indeed, true, and we summarize the changes in the following theorem.
\begin{theorem} \label{theorem:DBMDBMOU}
Suppose $X(t)$ solves the SDE in \eqref{eq:OUMatrixAugment}. Let $\{ \lambda_\alpha(t) \}_{\alpha}$ denote the eigenvalues of $X(t)$ realized as correlated paths on the simplex $\Delta_{M+N}$.
\begin{enumerate}
\item If $M = N$, then, the eigenvalues $\{ \lambda_\alpha(t) \}_\alpha$ solve the following system of SDEs:
\begin{align}
\dd\lambda_\alpha(t) \ = \ \frac{1}{\sqrt{N}} \dd B_\alpha(t) \ + \ \left( \frac{1}{2N} \sum_{\beta \neq \pm \alpha} \ \frac{1}{\lambda_\alpha - \lambda_\beta} - \frac{\lambda_\alpha(t)}{2} \right) \dd t,
\end{align}
where $\{ B_\alpha(t) \}_\alpha$ denote independent standard one-dimensional Brownian motions with the constraint $B_\alpha(t) = - B_{-\alpha}(t)$. Here, $B_{-\alpha}(t)$ denotes the Brownian motion driving the process defining $-\lambda_\alpha(t)$.
\item Conditioning on Assumption \ref{assumption:repulsionfrom0DBM}, if $M > N$, then the eigenvalues $\{ \lambda_\alpha(t) \}_{\alpha \not\in \zeta_t(X)}$ solve the following system of SDEs:
\begin{align}
\dd\lambda_\alpha(t) \ &= \ \frac{1}{\sqrt{N}} \dd B_\alpha(t) \ + \ \left( \frac{1}{2N} \sum_{\beta \not\in \zeta_t(X)}^{(\pm \alpha)} \ \frac{1}{\lambda_\alpha - \lambda_\beta} \ + \ \frac{M-N}{2N\lambda_\alpha}  - \frac{\lambda_\alpha(t)}{2} \right) \ \dd t \\
&= \ \frac{1}{\sqrt{N}} \dd B_\alpha(t) \ + \ \left( \frac{1}{2N} \sum_{\beta \not\in \zeta_t(X)}^{(\pm \alpha)} \ \frac{1}{\lambda_\alpha - \lambda_\beta} \ + \ \frac{1}{2N} \sum_{\gamma \in \zeta_t(X)} \ \frac{1}{\lambda_\alpha - \lambda_\gamma} - \frac{\lambda_\alpha(t)}{2} \right) \ \dd t
\end{align}
where the notation for the superscript $(\pm \alpha)$ over the summation indicates an omission of terms $\beta$ with $\beta = \pm \alpha$. Again, the $\{ B_\alpha(t) \}_\alpha$ denote independent standard one-dimensional Brownian motions except for the constraint $B_\alpha(t) = -B_{-\alpha}(t)$, where we retain the same notation for $B_{-\alpha}(t)$.
\end{enumerate}
\end{theorem}
\subsection{Derivation of the DBM equations}
As with the usual derivation of the DBM equations for Wigner matrices, we will apply the Ito formula. Exploiting the Ito formula will require eigenvalue-eigenvector identities known as the first- and second-order perturbation formulas, which we state after establishing the following notation.
\begin{notation}
Suppose $\pm \lambda_{\alpha} \not\in \zeta(X)$ are an eigenvalue pair of $X$, so in particular $-\lambda_{\alpha} = \lambda_{-\alpha}$. If $\mathbf{v}_{\alpha}$ denotes the eigenvector of $X$ corresponding to the eigenvalue $\lambda_{\alpha}$, then we let $\mathbf{v}_{-\alpha}$ denote the eigenvector of $X$ corresponding to the eigenvalue $\lambda_{-\alpha}$.
\end{notation}
By standard linear algebra (and/or compact operator theory), we know
\begin{align}
\mathbf{v}_{-\alpha}(i) \ = \ \begin{cases}
	\mathbf{v}_{\alpha}(i) & i \leq M \\
	-\mathbf{v}_{\alpha}(i) & i > M
\end{cases}. \label{eq:pmeigenvectors}
\end{align}
%
\begin{lemma} \label{lemma:perturbationformulasDBM}
Suppose $X \in \mathscr{M}_{\ell}$, and let $\{ \lambda_\alpha \}_\alpha$ denote the eigenvalues of $X$. The following formulas hold with high probability in context of the DBM. Writing $X = (X_{ij})$, we have, for $\lambda_\alpha(t) \not\in \zeta(X)$,
\begin{align}
\partial_{X_{ij}} \lambda_\alpha \ &= \ \mathbf{v}_\alpha(i) \mathbf{v}_\alpha(j), \\
\partial_{X_{ij} X_{k\ell}} \lambda_\alpha \ &= \ \sum_{\beta \neq \pm \alpha} \ \frac{\mathbf{v}_\beta(k) \mathbf{v}_\alpha(\ell) + \mathbf{v}_\beta(\ell) \mathbf{v}_\alpha(k) }{\lambda_\alpha - \lambda_\beta} \left[ \mathbf{v}_\beta(i) \mathbf{v}_\alpha(j) + \mathbf{v}_\beta(j) \mathbf{v}_\alpha(i) \right],
\end{align}
where the notation for the summation index denotes a sum over eigenvalues $\lambda_\beta \neq \pm \lambda_\alpha$. Moreover, the corresponding eigenvectors $\mathbf{v}_{\alpha}$ satisfy the following derivative formula:
\begin{align}
\partial_{X_{ij}} \mathbf{v}_\alpha \ = \ \sum_{\beta \neq \pm \alpha} \ \frac{\mathbf{v}_\beta^\ast \left[ \partial_{X_{ij}} X \right] \mathbf{v}_\alpha}{\lambda_\alpha - \lambda_\beta} \mathbf{v}_\beta.
\end{align}
\end{lemma}
The proof of these identities is standard and used in the derivation of the DBM equations for Wigner matrices, so we omit it. Before we proceed with the Ito formula calculation, we first introduce the following result, whose proof may be found in, for example, \cite{AGZ}. This result shows that the event on which the nontrivial eigenvalues of $X(t)$ intersect occurs with probability 0 and thus justifies the Ito formula calculation.
\begin{prop} \label{prop:intersectionprobability}
For each nontrivial eigenvalue $\lambda_{\alpha}(t) \not\in \zeta(X(t))$, define the event
\begin{align}
E_{\alpha}(t) \ := \ \bigcup_{\beta \neq \alpha: \lambda_{\beta} \not\in \zeta(X)} \ \left\{ \lambda_{\alpha}(t) = \lambda_{\beta}(t) \right\},
\end{align}
where $X(t)$ solves the stochastic matrix dynamics \eqref{eq:BMMatrixAugment} and \eqref{eq:OUMatrixAugment}. Then for any fixed time $T \geq 0$, 
\begin{align}
\mathbb{P} \left( \bigcap_{t \leq T} \ \bigcap_{\lambda_{\alpha} \not\in \zeta(X)} \ E_{\alpha}(t) \right) \ = \ 0.
\end{align}
\end{prop}
We begin by deriving the square DBM for the matrix-valued Brownian motion process. By Ito's formula, we have
\begin{align}
\d\lambda_\alpha \ &= \ \sum_{X_{ij} \in X} \ \partial_{X_{ij}} \lambda_\alpha \ \d X_{ij} \ + \ \frac12 \sum_{X_{ij}, X_{k\ell} \in X} \ \partial_{X_{ij} X_{k\ell}} \lambda_\alpha \ \d\langle X_{ij}, X_{k\ell} \rangle \\
&= \ \frac{1}{\sqrt{N}} \sum_{X_{ij} \in X} \ \partial_{X_{ij}} \lambda_\alpha \ \d B_{ij} \ + \ \frac12 \left( \sum_{X_{ij}, X_{k\ell} \in X} \ \partial_{X_{ij} X_{k \ell}} \lambda_\alpha \ \d\langle X_{ij}, X_{k\ell} \rangle \right).
\end{align}
In the above equation, we suppress from all processes the dependence on time. We first address the martingale term, which we denote by $\mathscr{X}_{\on{mgle}}(t)$. Using the first-order perturbation formula for the eigenvalue $\lambda_\alpha$, we have
\begin{align}
\mathscr{X}_{\on{mgle}}(t) \ = \ \frac{1}{\sqrt{N}} \left( \sum_{i = 1}^M \sum_{j = M+1}^{M+N} \ + \ \sum_{i = M+1}^{M+N} \sum_{j = 1}^N \right) \ \mathbf{v}_\alpha(i) \mathbf{v}_\alpha(j) \ \d B_{ij}.
\end{align}
We now let $\mathscr{X}_{\on{mgle},1}(t)$ denote the first sum in $\mathscr{X}_{\on{mgle}}(t)$ and $\mathscr{X}_{\on{mgle},2}(t)$ denote the second sum. 

Fix a time $t_0 \geq 0$, and let $\E_{t_0}$ denote conditional expectation conditioning on events occurring up to time $t_0$. We first compute
\begin{align}
\E_{t_0} \langle \mathscr{X}_{\on{mgle},1}(t), \mathscr{X}_{\on{mgle},1}(t) \rangle \ &= \ \frac{1}{N} \sum_{i = 1}^M \sum_{j = M+1}^{M+N} \ |\mathbf{v}_\alpha(i)|^2 |\mathbf{v}_\alpha(j)|^2 \langle \d B_{ij}, \d B_{ij} \rangle \\
&= \ \frac{1}{4N} \d t. 
\end{align}
This follows from the matrix structure of $X(t)$: for indices $(i,j)$ with $|j - i| \geq M$, the Brownian motions $\d B_{ij}$ are statistically independent. Similarly, 
\begin{align}
\E_{t_0} \langle \mathscr{X}_{\on{mgle},2}(t), \mathscr{X}_{\on{mgle},2}(t) \rangle \ &= \ \frac{1}{N} \sum_{i = M+1}^{M+N} \sum_{j = 1}^M \ |\mathbf{v}_\alpha(i)|^2 |\mathbf{v}_\alpha(j)|^2 \langle \d B_{ij}, \d B_{ij} \rangle \\
&= \ \frac{1}{4N} \d t.
\end{align}
Lastly, to compute the covariation, we have from the same calculation
\begin{align}
\E_{t_0} \langle \mathscr{X}_{\on{mgle},1}(t), \mathscr{X}_{\on{mgle},2}(t) \rangle \ &= \ \frac{1}{N} \sum_{i = 1}^{M} \sum_{j = 1}^{M} \ |\mathbf{v}_\alpha(i)|^2 |\mathbf{v}_\alpha(j)|^2 \ \d t \\
&= \ \frac{1}{4N} \d t.
\end{align}
Thus, we see that $\mathscr{X}_{\on{mgle}}(t)$ is a centered Gaussian process with quadratic variation given by 
\begin{align}
d \langle \mathscr{X}_{\on{mgle}}(t), \mathscr{X}_{\on{mgle}}(t) \rangle \ = \ \left( \frac{1}{4N} + \frac{1}{4N} + \frac{2}{4N} \right) \d t \ = \ \frac{1}{N} \d t.
\end{align}
This implies, by the Kolmogorov equations, that $\mathscr{X}_{\on{mgle}}(t)$ is a scaled Brownian motion, i.e.
\begin{align}
\mathscr{X}_{\on{mgle}}(t) \ = \ \frac{1}{\sqrt{N}} \d B_\alpha(t),
\end{align}
where $B_\alpha(t)$ is a standard one-dimensional Brownian motion on $\R$. To show the Brownian motions $B_\alpha(t)$ are independent, it suffices to show they are statistically uncorrelated because they are Gaussian random variables. This follows from orthonormality of the eigenvectors $\{ \mathbf{v}_\alpha\}_\alpha$ and the representation of $\mathscr{X}_{\on{mgle}}(t)$ in terms of the eigenvectors given in Lemma \ref{lemma:perturbationformulasDBM}. The relation $B_\alpha(t) = -B_{-\alpha}(t)$ follows immediately from $\dot{\lambda}_\alpha = -\dot{\lambda}_{-\alpha}$.

We now address the drift term; because the entries $X_{ij}, X_{k\ell}$ are driven by i.i.d. Brownian motions $B_{ij}(t), B_{k\ell}(t)$, we have
\begin{align}
\d \langle X_{ij}(t), X_{k\ell}(t) \rangle \ = \ \frac{1}{N} \left( \delta_{ik} \delta_{\ell j} + \delta_{i \ell} \delta_{jk} \right) \d t
\end{align}
given the symmetric structure of the matrix $X_t$. Thus, the drift term may be written as
\begin{align}
\mathscr{X}_{\on{drift}}(t) \ &= \ \frac12 \left( \frac{1}{N} \sum_{X_{ij} = X_{k\ell} \in X} \ \left( \partial_{X_{ij} X_{k \ell}} \lambda_\alpha \right) \right) \ \d t \\
&= \ \frac{1}{2N} \left( \sum_{i = 1}^M \ \sum_{j = M+1}^{M+N} \ \sum_{\beta \neq \pm \alpha} \ \frac{\mathbf{v}_\beta(i) \mathbf{v}_\alpha(j) + \mathbf{v}_\beta(j) \mathbf{v}_\alpha(i) }{\lambda_\alpha - \lambda_\beta} \left[ \mathbf{v}_\beta(i) \mathbf{v}_\alpha(j) + \mathbf{v}_\beta(j) \mathbf{v}_\alpha(i) \right] \right) \ \d t.
\end{align}
For each term in the sum, we expand the eigenvector terms and group according to indices $\alpha, \beta$ as follows:
\begin{align}
\left[ \mathbf{v}_\beta(i) \mathbf{v}_\alpha(j) + \mathbf{v}_\beta(j) \mathbf{v}_\alpha(i) \right] \left[ \mathbf{v}_\beta(i) \mathbf{v}_\alpha(j) + \mathbf{v}_\beta(j) \mathbf{v}_\alpha(i) \right] \ &= \ |\mathbf{v}_\beta(i)|^2 |\mathbf{v}_\alpha(j)|^2 + |\mathbf{v}_\beta(j)|^2 \mathbf{v}_\alpha(i)|^2 \nonumber \\
&\quad \ + \ 2 \mathbf{v}_\beta(i) \mathbf{v}_\alpha(i) \mathbf{v}_\beta(j) \mathbf{v}_\alpha(j).
\end{align}
By the spectral correspondence between $X(t)$ and the covariance matrices $X_{\ast}(t) = H(t)^{\ast} H(t)$ and $X_{\ast,+}(t) = H(t) H(t)^{\ast}$, we have
\begin{align}
\sum_{i = 1}^M \ \sum_{j = M+1}^{M+N} \ |\mathbf{v}_\beta(i)|^2 |\mathbf{v}_\alpha(j)|^2 \ = \ \sum_{i = 1}^M \ \sum_{j = M+1}^{M+N} \ |\mathbf{v}_\beta(j)|^2 |\mathbf{v}_\alpha(i)|^2 \ &= \ \frac12, \\
\sum_{i = 1}^M \ \mathbf{v}_\beta(i) \mathbf{v}_\alpha(i) \sum_{j = M+1}^{M+N} \ \mathbf{v}_\beta(j) \mathbf{v}_\alpha(j) \ &= \ 0,
\end{align}
where the second equation follows from orthonormality of different eigenvectors. Thus, we finally deduce
\begin{align}
\mathscr{X}_{\on{drift}}(t) \ = \ \frac{1}{2N} \sum_{\beta \neq \pm \alpha} \ \frac{1}{\lambda_\alpha - \lambda_\beta} \ \d t,
\end{align}
which completes the formal derivation of the square DBM. 

To derive the rectangular DBM, with the same argument as in the derivation of the square DBM, we have
\begin{align}
\d\lambda_\alpha \ = \ \frac{1}{\sqrt{N}} \d B_\alpha(t) \ + \ \frac{1}{2N} \sum_{i = 1}^M \ \sum_{j = M+1}^{M+N} \ \sum_{\beta \neq \pm \alpha} \ \partial_{X_{ij}}^{(2)} \lambda_\alpha \ \d t,
\end{align}
where the Brownian motions $B_\alpha(t)$ are independent. We now expand the drift term as follows:
\begin{align}
\mathscr{X}_{\on{drift}}(t) \ &= \ \frac{1}{2N} \sum_{i = 1}^M \ \sum_{j = M+1}^{M+N} \ \sum_{\beta \neq \pm \alpha} \ \partial_{X_{ij}}^{(2)} \lambda_\alpha \ \d t \\
&= \ \frac{1}{2N} \sum_{\beta \neq \pm \alpha} \ \sum_{i = 1}^M \ \sum_{j = M+1}^{M+N} \ \frac{\mathbf{v}_\beta(i) \mathbf{v}_\alpha(j) + \mathbf{v}_\beta(j) \mathbf{v}_\alpha(i)}{\lambda_\alpha - \lambda_\beta} \left[ \mathbf{v}_\beta(i) \mathbf{v}_\alpha(j) + \mathbf{v}_\beta(j) \mathbf{v}_\alpha(i) \right] \ \d t \\
&= \ \frac{1}{2N} \sum_{\beta \neq \pm \alpha} \ \frac{1}{\lambda_\alpha - \lambda_\beta} \ \d t.
\end{align}
We now note the following identity for any $\mathbf{v}_\beta$ with $\lambda_\beta \in \zeta(X)$, we have
\begin{align}
\sum_{i = 1}^M \ |\mathbf{v}_\beta(i)|^2 \ = \ 1
\end{align}
as well as the following vanishing identity for any index $j > M$:
\begin{align}
\mathbf{v}_\beta(j) \ = \ 0.
\end{align}
This follows from the spectral structure of the linearized covariance matrix $X(t)$; for details, we refer to \cite{Y1}. We now split the sum over eigenvalues $\lambda_\beta$ into those in and not in $\zeta(X)$, respectively:
\begin{align}
\mathscr{X}_{\on{drift}}(t) \ = \ \frac{1}{2N} \sum_{\beta \not\in \zeta(X)}^{(\pm \alpha)} \ \frac{1}{\lambda_\alpha - \lambda_\beta} \ + \ \frac{M-N}{2N \lambda_\alpha}.
\end{align}
Here, we used that any $\lambda \in \zeta(X)$ is equal to 0 and the size of $\zeta(X)$ is equal to $M-N$. Noting that by Assumption \ref{assumption:repulsionfrom0DBM} the eigenvalue $\lambda_\alpha$ is bounded uniformly away from 0, the second term coming from those $\lambda_\beta \in \zeta(X)$ does not diverge, completing the derivation of the rectangular DBM.

To address the Ornstein-Uhlenbeck variant, we note the Ito formula gives us
Proceeding as in the derivation of the square DBM and the rectangular DBM, and using the Ornstein-Uhlenbeck definition of $\d X_{ij}$ as given in \eqref{eq:OUMatrix}, we now have
\begin{align}
\d\lambda_\alpha \ &= \ \frac{1}{\sqrt{N}} \d B_\alpha(t) \ + \ \left( \frac{1}{2N} \sum_{\beta \not\in \zeta(X)}^{(\pm \alpha)} \ \frac{1}{\lambda_\alpha - \lambda_\beta} \ + \ \frac{1}{2N} \sum_{\gamma \in \zeta(X)} \ \frac{1}{\lambda_\alpha - \lambda_\gamma} \right) \ \d t \nonumber \\
&\quad \quad \ - \ \frac12 \sum_{X_{ij} \in X} \ X_{ij} \partial_{X_{ij}} \lambda_\alpha \ \d t,
\end{align}
where the contribution from the additional drift term in \eqref{eq:OUMatrix} is given by the term in the second line. By the first-order perturbation formula for the eigenvalue $\lambda_\alpha$, we see the contribution from this drift term is given by
\begin{align}
\frac12 \sum_{X_{ij} \in X} \ X_{ij} \partial_{X_{ij}} \lambda_\alpha \ &= \ \frac12 \sum_{i, j} \ \mathbf{v}_\alpha(i) X_{ij} \mathbf{v}_{\alpha}(j) \ = \ \frac12 \lambda_\alpha
\end{align}
since the first term on the RHS is exactly equal to the standard dot product of the eigenvector $\mathbf{v}_\alpha$ with $X\mathbf{v}_\alpha = \lambda_\alpha \mathbf{v}_\alpha$. This completes the derivation of Theorem \ref{theorem:DBMDBMOU}.
%
%
%
\section{Short-Range cutoff}
We now exploit the local structure of the DBM equations by introducing a scheme known as short-range approximation. To provide a brief summary, the short-range approximation is composed of the following steps:
\begin{itemize}
\item We begin letting the DBM equations run for a short-time $t_0$ to local equilibrium. At this point, the corresponding Stieltjes transforms of the particles $\{ \lambda_i(t_0) \}$ exhibit a strong local law and rigidity phenomena.
\item We then introduce cutoffs for the interaction terms. This will approximate the diffusion processes $\d \lambda_i(t)$ by a diffusion process blind to eigenvalues beyond a microscopic distance from $\lambda_i(t)$. 
\end{itemize}
Before we give the construction, we briefly motivate it by introducing a short-range approximation for Wigner matrices. To state this result, we refer to \cite{LSY} and \cite{LY} and introduce the DBM equations for Wigner matrices:
\begin{align}
\d z_i(t) \ = \ \sqrt{\frac{2}{N}} \d B_{i, \on{W}}(t) \ + \ \frac{1}{N} \sum_{j \neq i} \ \frac{1}{z_i(t) - z_j(t)} \ \d t. \label{eq:DBMWigner}
\end{align}
Here, the Brownian motions $B_{i, \on{W}}(t)$ are jointly independent with suitable initial data $z_i(0)$ as will be made precise later. We now introduce the short-range approximation result quite informally. This result is Lemma 3.7 in \cite{LSY}.
\begin{prop} \label{prop:cutoffWigner}
Suppose $z_i(t)$ solves \eqref{eq:DBMWigner} with initial data either deterministic with a $(g, G)$-regular potential, or distributed as the eigenvalues of a GOE matrix. Then if $\hat{z}_i(t)$ denotes the solution to the cutoff equation that will be described below for linearized covariance matrices, then the following estimate holds with $(\xi, \nu)$-high probability for constants $\omega_0, \omega_1, \omega_A, \omega_{\ell}$ and time $t_0$ to be determined:
\begin{align}
\sup_{0 \leq t \leq t_1} \ \sup_i \ \left| \hat{z}_i(t_0 + t) - z_i(t_0 + t) \right| \ \leq \ N^{\e} t_1 \left( \frac{N^{\omega_A}}{N^{\omega_0}} + \frac{1}{N^{\omega_{\ell}}} + \frac{1}{\sqrt{N G}} \right).
\end{align}
\end{prop}
Before we proceed, we first note the estimate in Proposition \ref{prop:cutoffWigner} includes a time-shift by $t_0$ as to allow for rigidity to hold. We now begin the short-range approximation scheme by introducing the following parameters:
\begin{align}
0 \ < \ \omega_1 \ < \ \omega_{\ell} \ < \ \omega_{A} \ < \ \frac{\omega_0}{2}.
\end{align}
%
\subsection{Time-Shift and Regularization}
We now renormalize the the DBM equations in time. To this end we define two time scales: the natural time scale of the DBM flow and the scale for which we allow the DBM to evolve afterwards.
\begin{align}
t_0 \ := \ N^{-1 + \omega_0}, \quad t_1 \ := \ N^{-1 + \omega_1}.
\end{align}
This motivates the following time-scale re-shift:
\begin{align}
\d z_i(t) \ := \ \d \lambda_i(t_0 + t).
\end{align}
As alluded to previously, we may now assume the particles $z_i(t)$ exhibit a rigidity phenomenon.
\subsection{Short-Range Cutoff}
We now introduce the cutoff in the drift/interaction terms. To this end, we first define the following which will help regularize the DBM. For a fixed energy $E$ contained in the interior of the support of the free convolution, we define the following classical location minimizing the distance from $E$:
\begin{align}
\gamma_E(t) \ := \ \underset{\gamma_{i,t}}{\on{arginf}} \ \left| E - \gamma_{i,t} \right|.
\end{align}
For notational convenience, we let $k(E)$ denote the index of the above minimizer, so that we have 
\begin{align}
\gamma_{E}(t) \ = \ \gamma_{k(E), t}.
\end{align}
With this, we define the following index set collecting eigenvalues for which we approximate by only short-range interactions:
\begin{align}
\mathscr{I}_{E, \omega_A} \ := \ \left\{ j: \ |j - k(E)| < N^{\omega_A} \right\}.
\end{align}
We also approximate the interval $qG$ on which the initial data is regular by classical locations of the free convolution law:
\begin{align}
\mathscr{C}_q \ := \ \left\{ j: \ \gamma_j \in \mathscr{I}_{E, qG} \right\}.
\end{align}
We now define the short-range cutoff for an individual eigenvalue $z_{i}(t)$ by collecting nearby eigenvalues via the set
\begin{align}
\mathscr{A}_{q} \ := \ \left\{ (i,j): \ |i - j| \ \leq \ N^{\omega_{\ell}} \right\} \ \bigcup \ \left\{ (i,j): \ ij > 0, \ i, j \not\in \mathscr{C}_{q} \right\}.
\end{align}
We briefly remark that the second set defining $\mathscr{A}_{q}$ serves to regularize the dynamics for eigenvalues outside the range of regularity for the potential $V$ defined by $\mathscr{I}_{E, qG}$. This will be made precise in the derivation of the short-range approximation. 

Before we introduce the short-range cutoffs, we establish the following notation to state the short-range cutoff more conveniently.
\begin{notation}
For a fixed index $i$, define the following summation operators:
\begin{align}
\sum_{j}^{\mathscr{A}_{q}(i)} \ := \ \sum_{j: (i,j) \in \mathscr{A}_{q}}, \quad \sum_{j}^{\mathscr{A}_{q}(i)^{C}} \ := \ \sum_{j: \ (i,j) \in \mathscr{A}_{q}^{C}}.
\end{align}
Here, the superscript $C$ denotes a set-theoretic complement.
\end{notation}
We now apply a deterministic shift to the DBM system as follows:
\begin{align}
\wt{z}_i(t) \ := \ z_i(t) \ - \ \gamma_E(t).
\end{align}
This deterministic shift is another regularization operator as it will be important in controlling the main error term in the short-range approximation. We note this shift is both deterministic and the same for all indices, so that local eigenvalue statistics should be preserved. By the inverse function theorem and differentiating the quantile representation of $\gamma_E(t)$, we compute the derivative of $\gamma_E(t)$ as follows:
\begin{align}
\partial_t \gamma_E(t) \ = \ - \on{Re} \left( m_{\on{fc},t}(\gamma_{E}(t)) \right) \ - \ \frac12 \gamma_E(t). \label{eq:dtclassicallocation}
\end{align}
The RHS in the derivative identity above is understood in the principal value sense; this is where we require $E$ lives in the interior of the support or completely separated from the support of the free convolution law. Thus, by the Ito formula we immediately deduce the following perturbed DBM equations for the shifted eigenvalues:
\begin{align}
\d \wt{z}_{i}(t) \ = \ \frac{1}{\sqrt{N}} \d B_i(t) \ + \ \left( \frac{1}{2N} \sum_{j \not\in \zeta_t(X)}^{(\pm i)} \ \frac{1}{\wt{z}_i(t) - \wt{z}_j(t)} \ + \ \frac{M-N}{N \wt{z}_i(t)} \ + \ \on{Re} \left( m_{\on{fc},t}(\gamma_{E}(t)) \right) \ + \ \frac12 \gamma_{E}(t) \right) \ \d t.
\end{align}
%
\subsection{Short-Range Equations}
We now define the short-range DBM equations as follows. For those indices $i \in \mathscr{I}_{E, \omega_A}$, we define the following short-range interaction equations:
\begin{align}
\d \hat{z}_i(t) \ := \ \frac{1}{\sqrt{N}} \d B_i(t) \ + \ \frac{1}{2N} \sum_{j}^{\mathscr{A}_{q}(i)} \ \frac{1}{\hat{z}_i(t) - \hat{z}_j(t)} \ \d t.
\end{align}
For those indices $i$ outside the interval $\mathscr{I}_{E, \omega_A}$, we define the following dynamics instead:
\begin{align}
\d \hat{z}_i(t) \ := \ \frac{1}{\sqrt{N}} &\d B_i(t) \nonumber \\
&\quad \ + \ \left( \frac{1}{2N} \sum_{j}^{\mathscr{A}_q(i)} \ \frac{1}{\hat{z}_i(t) - \hat{z}_j(t)} \ + \ \frac{1}{2N} \sum_{j}^{\mathscr{A}_{q}(i)^C} \ \frac{1}{\wt{z}_i(t) - \wt{z}_j(t)} \ + \ \on{Re} \left( m_{\on{fc},t}(\gamma_{E}(t)) \right) \ + \ \frac12 \gamma_{E}(t) \right) \ \d t.
\end{align}
Although not explicit, the summations in the drift term avoid the index $i$. To give initial conditions, for all indices we stipulate 
\begin{align}
\hat{z}_i(0) \ = \ \wt{z}_i(0) \ = \ \lambda_i(t_0) - \gamma_E(t_0).
\end{align}
We now derive and bound the error term in approximating the true DBM $\wt{z}_i(t)$ by the short-range dynamics $\hat{z}_i$. The following result shows that the error term is given by $N^{-1 - \delta}$ for small $\delta > 0$. This error is smaller than the scaling for eigenvalue gaps and thus suggests the local eigenvalue statistics for the true DBM and short-range DBM coincide in the limit of large $N$. We return to this point later, however, and proceed with the short-range approximation.
\begin{prop} \label{prop:shortrangeapproximation}
In the setting of the short-range approximation, for any fixed $\e > 0$ we have the following estimate with $(\xi, \nu)$-high probability for sufficiently large $N \gg 1$:
\begin{align}
\sup_{t \in [0, t_1]} \ \sup_{i} \ \left| \hat{z}_i(t) - \wt{z}_i(t) \right| \ \leq \ N^{\e} t_1 \left( \frac{N^{\omega_A}}{N^{\omega_0}} \ + \ \frac{1}{N^{\omega_{\ell}}} \ + \ \frac{1}{\sqrt{NG}} \right). \label{eq:shortrangeapproximation}
\end{align}
Here, we recall $G$ is the regularity parameter of the initial data, i.e. the scale on which rigidity holds.
\end{prop}
Before we proceed with the proof of Proposition \ref{prop:shortrangeapproximation}, we discuss both the result itself and its consequences. First, we deduce the following consequence which follows from Proposition \ref{prop:shortrangeapproximation} combined with Proposition \ref{prop:cutoffWigner}. The result also gives a rough idea of the proof of gap universality; we return to this idea shortly.
\begin{corollary} \label{corollary:covarianceWignercomparison}
With $(\xi, \nu)$-high probability for $N \gg 1$ sufficiently large, we have the following estimate:
\begin{align}
\sup_{t \in [0, t_1]} \ \sup_{i \in \mathscr{I}_{E, \omega_A}} \ \left| \lambda_i(t_0 + t) \ - \ \lambda_{W, i}(t_0 + t) \right| \ < \ N^{-1 - \delta}, \label{eq:covarianceWignercomparison}
\end{align}
where $\delta > 0$ is a small parameter depending on the parameters $\omega_0, \omega_1, \omega_A, \omega_{\ell}$. 
\end{corollary}
Indeed, Corollary \ref{corollary:covarianceWignercomparison} follows from the following eigenvalue SDEs for Wigner matrices of dimension $2N$ along short-time matrix-valued Brownian motion flows:
\begin{align}
\d \lambda_{W,i}(t) \ = \ \sqrt{\frac{2}{2N}} \d B_i(t) \ + \ \frac{1}{2N} \sum_{j}^{(i)} \ \frac{1}{\lambda_{W,i}(t) - \lambda_{W,j}(t)} \ \d t.
\end{align}
Here, the equation gives the honest dynamics without the short-range cutoff. Thus, with the respective cutoffs, the eigenvalues $\lambda_{W,i}$ and $\lambda_i$ solve the same SDEs. Thus, we obtain the estimate \eqref{eq:covarianceWignercomparison} with the bound on the RHS coming from the short-range approximation bounds in Proposition \ref{prop:cutoffWigner} and Proposition \ref{prop:shortrangeapproximation}, respectively. 

Moreover, in the same spirit we deduce the following reduction of the proof of gap universality, which states that it suffices to prove gap universality for the short-range dynamics.
\begin{corollary} \label{corollary:shortrangereduction}
In the context of Proposition \ref{prop:cutoffWigner} and Proposition \ref{prop:shortrangeapproximation}, suppose the gap universality estimate holds for the short-range eigenvalues $\hat{\lambda}$. Then the gap universality estimate holds for the full-range bulk eigenvalues $\lambda$.
\end{corollary}
We now record the following corollary of Proposition \ref{prop:shortrangeapproximation} which provides a weak form of level repulsion. This is a consequence of comparing the short-range DBM of linearized covariance matrices to the short-range DBM of Wigner matrices, i.e. \eqref{eq:covarianceWignercomparison}, as well as level repulsion estimates for Wigner matrices. Although this estimate is too weak for us to use in this thesis, we record it for possible future use.
\begin{corollary}
Suppose $\lambda_{i}(t)$ solves the DBM equations for linearized covariance matrices. Then with $(\xi, \nu)$-high probability, we have the following weak level repulsion for some small $\delta > 0$ and any index $i$ such that the classical location $\gamma_{i,t}$ is bounded away independent of $N$ from the boundary of the support of the free convolution $\varrho_{\on{fc},t}$:
\begin{align}
\left| \lambda_{i+1}(t) - \lambda_{i}(t) \right| \ \leq \ N^{-1 + \e}.
\end{align}
\end{corollary}
We briefly remark that an optimal level repulsion estimate seems to require analysis of the underlying matrix ensemble and exploiting its Gaussian component, which is the approach taken to obtain optimal level repulsion estimates for Wigner matrices. For a reference concerning level repulsion for Wigner matrices, see \cite{LY}.

We remark that the upper bound on the RHS of \eqref{eq:shortrangeapproximation} is heavily sensitive to the choice of parameters $\omega_0, \omega_1, \omega_{A}, \omega_{\ell}$. In particular, Proposition \ref{prop:shortrangeapproximation} is robust in the sense that the parameters are freely adaptable with differences by only small powers of $N$. However, the robustness of \eqref{eq:shortrangeapproximation} is limited by the quick deterioration of an $N^{-1 - \delta}$ estimate into a $\asymp N^{-1 + \delta}$ estimate.

Before we provide a proof of Proposition \ref{prop:shortrangeapproximation}, we will first make the assumption that $G \asymp 1$ is sufficiently large. In particular, the initial data is very regular, and rigidity holds for all eigenvalues. This will simplify the proof of Proposition \ref{prop:shortrangeapproximation}, and it applies to a wide variety of random matrix ensembles, e.g. biregular bipartite graphs. The proof for weaker regularity on the initial data, although possible, requires a technical and ad hoc argument in applying the strong local law and rigidity estimates. For details concerning initial potentials with less regularity, see Lemma 3.3 in \cite{LSY}.
\begin{proof}
(of Proposition \ref{prop:shortrangeapproximation}).

The key observation is the following dynamics for the difference term $w_i(t) := \hat{z}_i(t) - \wt{z}_i(t)$:
\begin{align}
\frac{\d}{\d t} w_i(t) \ = \ \frac{1}{2N} \sum_{j}^{\mathscr{A}_{q}(i)} \ \mathscr{B}_{ij}(t) \left( w_j(t) - w_i(t) \right) \ + \ \mathscr{E}_i. \label{eq:differencedynamics}
\end{align}
Here, the coefficients $\mathscr{B}_{ij}(t)$ and the error terms $\mathscr{E}_i(t)$ are given as follows:
\begin{align}
\mathscr{B}_{ij}(t) \ &= \ \frac{1}{(\hat{z}_i(t) - \hat{z}_j(t))(\wt{z}_{i}(t) - \wt{z}_{j}(t))}, \\
\mathscr{E}_{i}(t) \ &= \ \mathbf{1}_{i \in \mathscr{I}_{E, \omega_A}} \left[ - \frac{1}{2N} \sum_{j}^{\mathscr{A}_{q}(i)^C} \ \frac{1}{\wt{z}_i(t) - \wt{z}_j(t)} \ + \ \on{Re} \left( m_{\on{fc},t}(\gamma_E(t)) \right) \right]
\end{align}
Thus, by the Duhamel formula, we have
\begin{align}
\omega_i(t) \ = \ e^{t \mathscr{B}(t)} \omega_i(0) \ + \ \int_{0}^{t} \ e^{(t-s) \mathscr{B}(t - s)} \mathscr{E}_i(s) \ \d s.
\end{align}
Noting the dynamics of the difference terms $w_{i}$ are given by the dynamics of a jump process, $\mathscr{B}(t)$ denotes the associated semigroup. Because our initial condition $\omega_i(0)$ vanishes by construction, we have
\begin{align}
\omega_i(t) \ = \ \int_0^{t} \ e^{(t-s) \mathscr{B}(t-s)} \mathscr{E}_i(s) \ \d s.
\end{align}
Because $\mathscr{B}$ is the generator of a jump process on a discrete state space, it is a contraction on $\ell^{\infty}$ of the state space. This implies
\begin{align}
\| \omega_i(t) \|_{\ell^{\infty}} \ \leq \ \int_0^{t} \ \| e^{(t-s) \mathscr{B}(t-s)} \mathscr{E}_i(s) \|_{\ell^{\infty}} \ \d s \ \leq \ t \sup_{s \in [0,t]} \ \| \mathscr{E}_i(s) \|_{\ell^{\infty}}.
\end{align}
Thus, it remains to obtain an $\ell^{\infty}$ estimate on the error terms $\mathscr{E}_i$. To this end, we use the explicit representation of the error terms via the eigenvalue differences. In particular for $i \in \mathscr{I}_{E, \omega_A}$, we first rewrite the error term as follows:
\begin{align}
\mathscr{E}_i(t) \ &= \ \left[ - \frac{1}{2N} \sum_{j}^{\mathscr{A}_q(i)^{C}} \ \frac{1}{\wt{z}_i(t) - \wt{z}_j(t)} \ + \ \int_{\mathscr{I}_{E, t}(i)^C} \ \frac{\varrho_{\on{fc},t}(x)}{\wt{z}_i(t) - x} \ \d x \right] \\
&\quad \ + \ \left[ \int_{\mathscr{I}_{E,t}(i)^C} \ \frac{\varrho_{\on{fc},t}(x)}{\wt{z}_{i}(t) - x} \ \d x \ - \ \int_{\mathscr{I}_{E,t}(i)^C} \ \frac{\varrho_{\on{fc},t}(x)}{\gamma_{i,t} - x} \ \d x \right] \\
&\quad \ + \ \left[ \on{Re} \left( m_{\on{fc},t}(\gamma_E(t)) \right) \ - \ \on{Re} \left( m_{\on{fc},t}(\gamma_{i,t}) \right) \right] \ + \ \left[ \gamma_{i,t} - \gamma_E(t) \right] \\
&\quad \ + \ \int_{\mathscr{I}_{E, t}(i)} \ \frac{\varrho_{\on{fc},t}(x)}{\gamma_{i,t} - x} \ \d x \\
&= \ \mathscr{F}_{1}(t) \ + \ \mathscr{F}_{2}(t) \ + \ \mathscr{F}_{3}(t) \ + \ \mathscr{F}_{4}(t).
\end{align}
Here, the integral in defining the error term $\mathscr{F}_{4}(t)$ is understood in the sense of principal values, as are the Stieltjes transform terms defining the error term $\mathscr{F}_{3}(t)$. It is now our goal to bound each of the error terms above.

For the first error term, we use the strong local law which holds with $(\xi, \nu)$-high probability for all times:
\begin{align}
\mathscr{F}_1(t) \ &= \ \int_{\mathscr{I}_{E,t}(i)^C} \ \frac{1}{\wt{z}_i(t) - x} \left( \varrho_{\on{fc},t}(x) \ - \ \frac{1}{2N} \sum_{\lambda \in \sigma(X(t))} \ \delta(x - \lambda)\right) \ \d x \ + \ O(N^{-1 + \delta}) \\
&\leq \ C \int_{\mathscr{I}_{E,t}(i)^{C}} \ \frac{1}{\wt{z}_i(t) - x + i N^{\omega_{\ell} - 1 - \delta}} \left( \varrho_{\on{fc},t}(x) \ - \ \frac{1}{2N} \sum_{\lambda \in \sigma(X(t))} \ \delta(x - \lambda)\right) \ \d x  \ + \ O(N^{-1 + \delta}).
\end{align}
Here, $\delta$ is an arbitrarily small but fixed constant. We note the big-Oh term comes from rigidity of eigenvalues whose corresponding classical locations are with $N^{-1 + \delta}$ of the boundary of $\mathscr{I}_{E,t}(i)$. The second inequality is taken in the sense of absolute values. This inequality also follows from rigidity, since $\wt{z}_i(t)$ is separated from the boundary of $\mathscr{I}_{E,t}(i)$ by a distance bounded below by $N^{\omega_{\ell} - 1 - \delta}$, upon possibly redefining $\delta > 0$. With this, the strong local law implies, for some $\e > 0$,
\begin{align}
\mathscr{F}_1(t) \ \leq \ \frac{N^{\e}}{N^{\omega_{\ell}}}.
\end{align}
We now estimate the second error term similarly using rigidity:
\begin{align}
\mathscr{F}_2(t) \ &= \ \int_{\mathscr{I}_{E,t}(i)^{C}} \ \varrho_{\on{fc},t}(x) \frac{\gamma_{i,t} - \wt{z}_i(t)}{(\wt{z}_i(t) - x)(\gamma_{i,t} - x)} \ \d x \\
&\leq \ \frac{N^{1 + \e}}{N^{\omega_{\ell}}} \int_{\mathscr{I}_{E,t}(i)^{C}} \ \varrho_{\on{fc},t}(x) \frac{N^{-1 + \e}}{|\wt{z}_i(t) - x| + i N^{\omega_{\ell} - 1 - \delta}}  \ \d x \\
&\leq \ \frac{N^{\e}}{N^{\omega_{\ell}}} \int_{\mathscr{I}_{E, t}(i)^{C}} \ \frac{\varrho_{\on{fc},t}(x)}{\left| \wt{z}_i(t) - x \right| + i N^{\omega_{\ell} - 1 - \delta}} \ \d x.
\end{align}
To estimate this last term, because $\varrho_{\on{fc},t}$ has compact support, it suffices to estimate the Stieltjes transform of the free convolution for large energy. To this end we use a general result for Stieltjes transforms of compactly supported measures:
\begin{align}
\left| m_{\on{fc},t}(E + i \eta) \right| \ \leq \ C \log(N)^{C} \sup_{\eta' \geq \eta} \ \on{Im} \left( m_{\on{fc},t}(E + i \eta') \right).
\end{align}
Here, $C > 0$ is a fixed constant. For a proof of this result, we refer to Lemma 7.1 in \cite{LY}.

By the fixed-point equation defining the free convolution, the RHS is bounded by $C \log(N)^C$. Thus, we deduce the following bound on the second error term for some possibly adapted constant $\e > 0$:
\begin{align}
\left| \mathscr{F}_2(t) \right| \ \leq \ C \frac{N^{\e}}{N^{\omega_{\ell}}}.
\end{align}
We now estimate the third error term with the following bound for an arbitrarily small but fixed $\delta > 0$:
\begin{align}
\left| \mathscr{F}_3(t) \right| \ \leq \ C \left( \frac{N^{\omega_{A}}}{N^{\omega_0}} \ + \ \frac{N^{\omega_A}}{N} \right).
\end{align}
This follows from bounding the following time-derivative of the Stieltjes transform given in Lemma 3.3 in \cite{LSY}:
\begin{align}
\left| \frac{\partial}{\partial z} m_{\on{fc},t_0 + t}(z) \right| \ \leq \ \frac{C}{t_0} \ \asymp \ \frac{C}{N^{-1 + \omega_0}}. \label{eq:derivativeST}
\end{align}
Indeed, by construction, we know $|\gamma_{i,t} - \gamma_E(t)|$ is bounded by $N^{\omega_A - 1}$ up to a constant independent of $N$. This gives the desired estimate:
\begin{align}
\left| \mathscr{F}_{3}(t) \right| \ \leq \ C \left( \frac{N^{\omega_A - 1}}{N^{\omega_0 - 1}} \ + \ \frac{N^{\omega_A}}{N} \right).
\end{align}
It remains to bound the last error term. To this end, we first note the interval $\mathscr{I}_{E, t}(i)$ is almost symmetric by definition. To elaborate, if $\mathscr{I}_{E, t}(i)$ were symmetric about $\gamma_{i}(t)$, then the error term $\mathscr{F}_{4}(t)$ would be bounded by the $C^1$-norm of the free convolution law up to a constant by the classical regularization procedure of the principal value. However, we instead have the estimate
\begin{align}
\left| \mathscr{F}_4(t) \right| \ = \ \left| \int_{\mathscr{I}_{E,t}(i)} \ \frac{\varrho_{\on{fc},t}(x)}{\gamma_{i,t} - x} \ \d x \right| \ &\leq \ C \| \nabla \varrho_{\on{fc},t} \|_{C^0} \ + \ C \frac{N^{\omega_{\ell}}}{N t_0}.
\end{align}
Indeed, the first term comes from regularizing the integral on some small symmetric interval centered at $\gamma_{i,t}$ contained in $\mathscr{I}_{E,t}(i)$, and the second term follows from a simple bound using the following estimate on classical locations which in turn follows from the derivative estimate \eqref{eq:derivativeST}:
\begin{align}
\gamma_{i + k, t_0 + t} - \gamma_{i, t_0 + t} \ = \ \left( \gamma_{i - k, t_0 + t} - \gamma_{i, t_0 + t} \right) \left( 1 + O \left( \frac{k}{N t_0} \right) \right).
\end{align}
We may also bound the derivative norm by applying the inverse Stieltjes transform to \eqref{eq:derivativeST}. This implies
\begin{align}
\left| \mathscr{F}_{4}(t) \right| \ \leq \ C \frac{N^{\omega_{\ell}}}{N^{\omega_0}}.
\end{align}
Combining the bounds for all the error terms, we immediately deduce the desired estimate \eqref{eq:shortrangeapproximation}.
\end{proof}
Before we proceed, we note our above estimates did not produce a term on the order of $(NG)^{-1/2}$. This is because we assumed high regularity, i.e. $G \asymp 1$, so a direct application of the strong local law and rigidity cannot see this term. We note, however, this missing error term appears naturally in handling the error term $\mathscr{F}_1$ assuming little regularity on the initial data. For details we again refer to Section 3 in \cite{LSY}.
%
%
%
\section{Gap Universality: Proof of the Theorem \ref{theorem:universalgaps}}
We now use the short-range approximation in Proposition \ref{prop:shortrangeapproximation} to derive gap universality. The method here is taken from the proof of gap universality for the Wigner ensemble given in \cite{LY}. We begin by introducing the coupled Dyson Brownian Motions matching the process $\hat{\lambda}_i(t)$ with GOE eigenvalue dynamics. To this end, recall that $\hat{\lambda}_{i}(t)$ denotes the solutions to the following short-range DBM equations:
\begin{align}
\d \hat{\lambda}_{i}(t_0 + t) \ = \ \frac{1}{\sqrt{N}} \d B_{i}(t_0 + t) \ + \ \frac{1}{2N} \sum_{j}^{\mathscr{A}_{q}(i)} \ \frac{1}{\hat{\lambda}_{i}(t_0 + t) - \hat{\lambda}_{j}(t_0 + t)} \d t. 
\end{align}
We emphasize the short-range interactions in the drift term contain at most one contribution from each pair of totally anticorrelated eigenvalues $(\pm \lambda_i)$ in the spectrum of $X(t)$.
\subsection{The Coupled Gaussian Dynamics}
We now fix an energy $E$ in the bulk of the linearized Marchenko-Pastur law and an index $i \in \mathscr{I}_{E, \omega_A}$. We now define the following matrix:
\begin{align}
W(t_0) \ = \ a_0 W \ + \ b_0, \quad \ a_0 \ = \ \frac{\varrho_{\on{sc}}(\mu_i)}{\varrho_{\on{fc},t_0}(\gamma_{i,t_0})}, \ b_0 \ = \ \gamma_{i,t_0} - a_0 \mu_i.
\end{align}
We now define the DBM flow for the eigenvalues of a GOE matrix with initial data given by the eigenvalues of $W(t_0)$. We first define the short-range DBM denoted by $\hat{\nu}(t)$ given by the following system of SDEs for $i \in \mathscr{I}_{E, \omega_A}$:
\begin{align}
\d \hat{\nu}_{i}(t_0 + t) \ = \ \frac{1}{\sqrt{N}} \d B_{i}(t_0 + t) \ + \ \frac{1}{2N} \sum_{j}^{\mathscr{A}_{q}(i)} \ \frac{1}{\hat{\nu}_i(t_0 + t) - \hat{\nu}_j(t_0 + t)} \ \d t. \label{eq:shortrangeGOE}
\end{align}
Here, the Brownian motions defining \eqref{eq:shortrangeGOE} are the same Brownian motions driving the short-range process $\hat{\lambda}(t)$. 

We now define the full-range DBM with initial data is the full spectrum of $W_{a,b}(t_0)$. We will not study the analysis of the full-range DBM in detail, as it will only be necessary in knowing the global structure of the eigenvalue system. We define this full-range DBM though the following SDEs:
\begin{align}
\d \nu_i(t_0 + t) \ = \ \frac{1}{\sqrt{N}} \d B_i(t_0 + t) \ + \ \frac{1}{2N} \sum_{j}^{(i)} \ \frac{1}{\nu_i(t_0 + t) - \nu_j(t_0 + t)} \ \d t.
\end{align}
Because the Gaussian data is invariant under the full-range DBM equations above, we deduce the system $\nu(t)$ gives the spectrum of a sub-linear Gaussian perturbation of $W(t_0)$. This may be realized as processes $a_t$ and $b_t$ with initial data $a_0$ and $b_0$ defined above; this is the desired global data mentioned in introducing the full-range DBM.

\subsection{Gap Universality}
We now use the coupled short-range GOE dynamics $\hat{\nu}(t)$ to deduce gap universality for the linearized covariance matrix ensemble. By Proposition \ref{prop:shortrangeapproximation} we have the following raw gap estimates for the short-range dynamics with $(\xi, \nu)$-high probability:
\begin{align}
\sup_{t \in [0, t_1]} \ \left| \left( \hat{\lambda}_{i}(t_0 + t) - \hat{\lambda}_{i+k}(t_0 + t) \right) \ - \ \left( \hat{\nu}_{i}(t_0 + t) - \hat{\nu}_{i+k}(t_0 + t) \right) \right| \ < \ N^{-1 - \e}.
\end{align}
Here, $\e > 0$ is a small, fixed constant. Because $\varrho_{\on{fc},t_0} \asymp 1$ in the bulk, the same gap estimate holds multiplying the LHS by $\varrho_{\on{fc},t_0}(\gamma_{i,t_0})$, upon possibly adjusting the constant $\e > 0$. This gives the following preliminary estimate upon a possible readjustment of the parameter $\e > 0$:
\begin{align}
\sup_{t \in [0, t_1]} \ \left| \varrho_{\on{fc},t_0}(\gamma_{i,t_0})  \left( \hat{\lambda}_{i}(t_0 + t) - \hat{\lambda}_{i+k}(t_0 + t) \right) \ - \ \varrho_{\on{fc},t_0}(\gamma_{i,t_0}) \left( \hat{\nu}_{i}(t_0 + t) - \hat{\nu}_{i+k}(t_0 + t) \right)\right| \ < \ N^{-1 - \e}. \label{eq:preliminarybound}
\end{align}
This, although close, is not the desired bound from which we may deduce gap universality from a first-order Taylor expansion. Through perturbative methods, however, we may deduce the desired bound from \eqref{eq:preliminarybound}; in particular we aim to perturb the free convolution factors and the classical locations in time, which requires a time- and energy- derivative estimate on the free convolution. We make this more precise shortly.
\begin{lemma} \label{lemma:desiredgapestimate}
In the context of the short-range dynamics $\hat{\lambda}(t)$ and $\hat{\nu}(t)$, we have
\begin{align}
\sup_{t \in [0, t_1]} \ \left| \varrho_{\on{fc},t}(\gamma_{i,t})  \left( \hat{\lambda}_{i}(t_0 + t) - \hat{\lambda}_{i+k}(t_0 + t) \right) \ - \ \frac{\varrho_{\on{sc}}(\mu_i)}{a_t} \left( \hat{\nu}_{i}(t_0 + t) - \hat{\nu}_{i+k}(t_0 + t) \right)\right| \ < \ N^{-1 - \e}.
\end{align}
\end{lemma}
\begin{proof}
To replace this free convolution factor at $t_0$ by the free convolution at time $t_0 + t$, we appeal to the following simple estimate:
\begin{align}
\left| \varrho_{\on{fc},t}(\gamma_{i,t}) \ - \ \varrho_{\on{fc},t_0}(\gamma_{i, t_0}) \right| \ &\leq \ \left| \varrho_{\on{fc},t}(\gamma_{i,t}) \ - \ \varrho_{\on{fc},t}(\gamma_{i,t_0}) \right| \ + \ \left| \varrho_{\on{fc},t}(\gamma_{i,t_0}) \ - \ \varrho_{\on{fc},t_0}(\gamma_{i,t_0}) \right| \\
&\leq \ C \frac{N^{\omega_1} \log N}{N^{\omega_0}} \ + \ \left| \varrho_{\on{fc},t}(\gamma_{i,t_0}) \ - \ \varrho_{\on{fc},t_0}(\gamma_{i,t_0}) \right|.
\end{align}
Here, the second inequality follows from an energy-derivative estimate on the free convolution density given by applying the Stieltjes inversion formula to \eqref{eq:derivativeST} and bounding the time derivative of the classical location computed in \eqref{eq:dtclassicallocation} by $\log N$. It remains to bound the second term in the inequality above. We do so by bounding the time-derivative of the free convolution density. Indeed, we may obtain the following time-derivative bound by computing the time-derivative of its Stieltjes transform via its fixed-point equation and applying the Stieltjes inversion formula. This method yields the following estimate; for a reference, see Lemma 7.6 in \cite{LY}:
\begin{align}
\left| \partial_t \varrho_{\on{fc},t_0}(E) \right| \ \leq \ \frac{C}{t_0} \ = \ \frac{C}{N^{-1 + \omega_0}}.
\end{align}
Because $t - t_0 \leq t_1 = N^{-1 + \omega_1}$, we ultimately deduce
\begin{align}
\left| \varrho_{\on{fc},t}(\gamma_{i,t}) \ - \ \varrho_{\on{fc},t_0}(\gamma_{i, t_0}) \right| \ \leq \ C \frac{N^{\omega_1}}{N^{\omega_0}} \left( \log N + 1 \right) \ \leq \ \frac{1}{N^{\omega_2}},
\end{align}
for some fixed constant $\omega_2 > 0$. Combining our estimates thus far, we have the following bound with $(\xi, \nu)$-high probability uniformly over all times $t \in [t_0, t_0+t_1]$:
\begin{align}
\left| \varrho_{\on{fc},t}(\gamma_{i,t}) \left( \lambda_{i} - \lambda_{i+k} \right) \ - \ \frac{\varrho_{\on{sc}}(\mu_i)}{a_0}\left( \nu_{i} - \nu_{i+k} \right) \right| \ \leq \ N^{-1 - \e}, \label{eq:almosttheregaps}
\end{align}
upon possibly redefining $\e > 0$. Here, we keep the $o(N^{-1})$ estimate on the RHS by appealing to the following bound which holds by rigidity:
\begin{align}
\left| \varrho_{\on{fc},t}(\gamma_{i,t}) \ - \ \varrho_{\on{fc},t_0}(\gamma_{i,t_0}) \right| \left( \lambda_{i} - \lambda_{i+k} \right) \ \leq \ N^{-\omega_2} N^{-1 + \delta} \ \leq \ N^{-1 - \e}.
\end{align}
Here, $\delta > 0$ is arbitrarily small and $\e > 0$ is fixed. However, because eigenvalues differences change at most linearly upon linear perturbations of the underlying matrix, we deduce the simple bound $|a_t - a_0| \leq Ct$. This implies the following time-adapted bound on the event on which the upper bound \eqref{eq:almosttheregaps} holds:
\begin{align}
\left| \varrho_{\on{fc},t}(\gamma_{i,t}) \left( \lambda_{i} - \lambda_{i+k} \right) \ - \ \frac{\varrho_{\on{sc}}(\mu_i)}{a_t}\left( \nu_{i} - \nu_{i+k} \right) \right| \ \leq \ N^{-1 - \e}. \label{eq:almosttherewowgaps}
\end{align}
This follows from the constraint that we stipulate $t \leq t_1$, which concludes the proof.
\end{proof}
We now make the following observation; by construction the gaps $a_t^{-1}(\nu_i - \nu_{i+k})$ are distributed as the gaps of a standard GOE matrix. Moreover, by applying the proof of Lemma \ref{lemma:desiredgapestimate} iteratively $n$ times for any $n = O(1)$, we also deduce the following multi-gap estimate with $(\xi, \nu)$-high probability:
\begin{align}
\sup_{k \in [0,n]} \ \sup_{t \in [0, t_1]} \ \left| \varrho_{\on{fc},t}(\gamma_{i,t})  \left( \hat{\lambda}_{i}(t_0 + t) - \hat{\lambda}_{i+k}(t_0 + t) \right) \ - \ \frac{\varrho_{\on{sc}}(\mu_i)}{a_t} \left( \hat{\nu}_{i}(t_0 + t) - \hat{\nu}_{i+k}(t_0 + t) \right)\right| \ < \ N^{-1 + \e}. \label{eq:taylorexpandthis}
\end{align}
From here, to deduce gap universality we Taylor expand our test function $O \in C^{\infty}_c(\R^n)$ on the event on which \eqref{eq:taylorexpandthis} holds. In particular, by a first-order Taylor expansion on this event we deduce the following straightforward estimate:
\begin{align}
O &\left( N \varrho_{\on{fc},t}(\gamma_{i,t}) (\lambda_{i} - \lambda_{i + i_1} ), \ldots, N \varrho_{\on{fc},t}(\gamma_{i,t}) (\lambda_{i} - \lambda_{i + i_n}) \right) \nonumber \\
&\quad \ - \ O \left( N \varrho_{\on{sc}}(\mu_i) (\lambda_{i} - \lambda_{i+i_1}), \ldots, N \varrho_{\on{sc}}(\mu_i)(\lambda_{i} - \lambda_{i + i_n} ) \right) \ \lesssim_{\| O \|_{C^1}, n} \ N^{-\e},
\end{align}
where $\e > 0$ is the same exponent in the estimate \eqref{eq:almosttherewowgaps}.

On the complement of the event on which \eqref{eq:almosttherewowgaps} holds, we may apply the following crude bound which follows from the spectral bound on the initial data $V$:
\begin{align}
O &\left( N \varrho_{\on{fc},t}(\gamma_{i,t}) (\lambda_{i} - \lambda_{i + i_1} ), \ldots, N \varrho_{\on{fc},t}(\gamma_{i,t}) (\lambda_{i} - \lambda_{i + i_n}) \right) \nonumber \\
&\quad \ - \ O \left( N \varrho_{\on{sc}}(\mu_i) (\lambda_{i} - \lambda_{i+i_1}), \ldots, N \varrho_{\on{sc}}(\mu_i)(\lambda_{i} - \lambda_{i + i_n} ) \right) \ \lesssim_{\| O \|_{\mathscr{S}}, n} \ N^{C}
\end{align}
for any $C = O(1)$. However, because \eqref{eq:almosttherewowgaps} holds with $(\xi, \nu)$-high probability, by taking expectations we deduce
\begin{align}
\E^{X(t)} &\left[ O \left( N \varrho_{\on{fc},t}(\gamma_{i,t}) (\lambda_{i} - \lambda_{i + i_1} ), \ldots, N \varrho_{\on{fc},t}(\gamma_{i,t}) (\lambda_{i} - \lambda_{i + i_n}) \right) \right] \nonumber \\
&- \ \E^{\on{GOE}} \left[ O \left( N \varrho_{\on{sc}}(\mu_i) (\lambda_{i} - \lambda_{i+i_1}), \ldots, N \varrho_{\on{sc}}(\mu_i)(\lambda_{i} - \lambda_{i + i_n} ) \right) \right]  \ \leq \ N^{-\e} \ + \ N^C e^{-\nu \log^{\xi} N}.
\end{align}
Because $\xi \gg 1$, appealing to Corollary \ref{corollary:shortrangereduction}, this completes the proof of Theorem \ref{theorem:universalgaps}.
%
%
%
\section{Appendix}
\subsection{Concentration estimates and stochastic continuity}
We now focus on extending the strong local law in Theorem \ref{theorem:stronglocallawstronglocallawdeformedmatrices} from fixed times in $\mathbf{T}_{\omega}$ to all times simultaneously with high probability. We begin with the following concentration estimate for sub-exponential random variables; for a reference, see \cite{LY}.
\begin{prop} \label{prop:concentration}
Suppose $(a_i)_{i \in I}$ and $(b_i)_{i \in I}$ are independent families of centered random variables with variance $\sigma^2$ and sub-exponential decay in the following sense:
\begin{align}
\mathbb{P} \left( |a_i| \geq x \sigma \right) + \mathbb{P} \left( |b_i| \geq x \sigma \right) \ \leq \ C_0 e^{-x^{1/\theta}},
\end{align}
where $C_0 > 0$ and $\theta \geq 1$ are constants uniform over $i \in I$. Suppose $A_i, B_{ij} \in \C$ are deterministic constants. Then, for some constants $a_0 > 1, A_0 \geq 10$, and $C_1 \geq 1$ depending on $\theta, C_0$, for any parameter $a_0 \leq \xi \leq A_0 \log \log N$ and $\varphi = \log^{C_1} N$, we have the following concentration inequalities:
\begin{align}
\mathbb{P} \left( \left| \sum_{i = 1}^N \ A_i a_i \right| \geq \varphi^{\xi} \sigma \left( \sum_{i = 1}^N \ |A_i|^2 \right)^{1/2} \right) \ &\leq \ e^{- \log^\xi N}, \\
\mathbb{P} \left( \left| \sum_{i = 1}^N \ \overline{a_i} B_{ii} a_i - \sum_{i = 1}^N \sigma^2 B_{ii} \right| \geq \varphi^\xi \sigma \left( \sum_{i = 1}^N |B_{ii}|^2 \right)^{1/2} \right) \ &\leq \ e^{-\log^\xi N}, \\
\mathbb{P} \left( \left| \sum_{i \neq j}^N \ \overline{a_i} B_{ij} a_j \right| \geq \varphi^\xi \sigma \left( \sum_{i \neq j} \ |B_{ij}|^2 \right)^{1/2} \right) \ &\leq \ e^{-\log^\xi N}, \\
\mathbb{P} \left( \left| \sum_{i,j = 1}^N \ \overline{a_i} B_{ij} b_j \right| \geq \varphi^\xi \sigma \left( \sum_{i,j = 1}^N \ |B_{ij}|^2 \right)^{1/2} \right) \ &\leq \ e^{-\log^\xi N}.
\end{align}
\end{prop}
We now present a general result in stochastic calculus which states that Ito integrals of bounded functions are indeed sub-exponential. This result follows from the Burkholder inequality applied to a Brownian integral.
\begin{prop} \label{prop:subexponentialIto}
Let $t > 0$ be a positive time and assume $f(x) \in C^{2}(\R)$ is bounded, adapted and measurable. Suppose $X_t$ is given by the following Ito integral:
\begin{align}
X_t \ = \ \int_0^t \ f(X_s) \ \dd B_s.
\end{align}
Then $\sup_{t \leq T} X_t$ is sub-exponential, and thus with $(\xi, \nu)$-high probability 
\begin{align}
\sup_{t \leq T} \ |X_t| \ \leq \ C \sigma_T 
\end{align}
for some constant $C= O(1)$. Here, $\sigma_T$ is the standard deviation of $X_T$. 
\end{prop}
With these two preliminary, general results, we now compute the SDE dynamics of the Stieltjes transform $m_N(z; t)$ via the Ito formula.
\begin{prop}
Let $m_N(z; t)$ denote the Stieltjes transform of $X(t)$. Moreover, suppose $\eta = \on{Im}(z) > 0$. Then, for any times $0 \leq t_0 \leq t_1$, we have
\begin{align}
m_N(z; t_1) - m_N(z; t_0) \ &= \ \sum_{\lambda_\alpha \not\equiv 0} \ \frac{1}{\sqrt{N}} \int_{t_0}^{t_1} \ \frac{1}{(\lambda_\alpha - z)^2} \ \dd B_\alpha(t) \nonumber \\
&\quad \ + \ \left( \frac{1}{N} \sum_{\lambda_\alpha \not\equiv 0} \ \frac{1}{(\lambda_\alpha - z)^3} \ + \ \frac{1}{2N} \sum_{\lambda_\alpha \not\equiv \pm \lambda_\beta} \ \frac{\lambda_\alpha + \lambda_\beta - 2z}{(\lambda_\alpha - z)^2(\lambda_\beta - z)^2} \ + \ \frac{M-N}{\lambda_\alpha} \right) \ \dd t.\label{eq:SDESTdeformedmatrices}
\end{align}
\end{prop}
\begin{proof}
Because the Brownian motions $B_{\alpha}(t)$ are jointly independent, they are also orthogonal, i.e.
\begin{align}
\d \langle \lambda_{\alpha}(t), \lambda_{\beta}(t) \rangle \ = \ \frac{\delta_{\alpha \beta}}{N} \d t.
\end{align}
From this, by the Ito formula, at least formally we have
\begin{align}
m_N(z; t_1) - m_N(z; t_0) \ &= \ \sum_{\lambda_{\alpha} \not\equiv 0} \ \int_{t_0}^{t_1} \ \partial_{\lambda_{\alpha}} m_N(z; t) \d \lambda_{\alpha} \ + \ \frac{1}{2N} \sum_{\lambda_{\alpha} \not\equiv 0} \ \partial_{\lambda_{\alpha}}^{(2)} m_N(z;t) \ \d t.
\end{align}
From here, it suffices to calculate the derivatives of $m_N$ with respect to the nontrivial eigenvalues $\lambda_{\alpha} \not\equiv 0$; because $m_N$ is nonsingular at $z = E + i \eta \in \C_+$, the following derivative identities hold rigorously as well:
\begin{align}
\partial_{\lambda_\alpha} \ m_N(z; t) \ &= \ \sum_{\lambda_\alpha} \ - \frac{1}{(\lambda_\alpha - z)^2}, \\
\partial_{\lambda_\alpha}^{(2)} \ &= \ \frac{1}{2} \sum_{\lambda_\alpha} \ \frac{1}{(\lambda_\alpha - z)^3}.
\end{align}
Lastly, we remark that the calculation with Ito's formula holds rigorously at $z = E + i \eta \in \C_+$, a point at which $m_N$ is smooth as a function of the real eigenvalues $\lambda_{\alpha}$.
\end{proof}
With the martingale bound in Proposition \ref{prop:subexponentialIto}, we now derive the following short-time estimate for the SDE dynamics of the Stieltjes transform $m_N(z; t)$.
\begin{lemma} \label{lemma:SDESTdeformedmatrices}
Suppose $0 \leq t_0 \leq t_1$ are two times. Then we have uniformly over $z \in \mathbf{D}_{L,q}$ the following bound with $(\xi, \nu)$-high probability:
\begin{align}
\sup_{t \in [t_0, t_1]} \ \left| m_N(z; t) - m_N(z; t_0) \right| \ = \ O \left( (t_1 - t_0) \left( \frac{1}{\eta^3} + \frac{N^{B_V}}{\eta^4} + N^{5B_V} + 1 + \frac{1}{\eta^4} \right) \right).
\end{align}
\end{lemma}
\begin{proof}
We first consider the drift term in the SDE \eqref{eq:SDESTdeformedmatrices}. Suppose $z \in \mathbf{D}_1$, so that the drift term is bounded as follows:
\begin{align}
\left| \frac{1}{N} \sum_{\lambda_\alpha \not\equiv 0} \ \frac{1}{(\lambda_\alpha - z)^3} \ + \ \frac{1}{2N} \sum_{\lambda_\alpha \not\equiv \pm \lambda_\beta} \ \frac{\lambda_\alpha + \lambda_\beta - 2z}{(\lambda_\alpha - z)^2(\lambda_\beta - z)^2} \ + \ \frac{M-N}{\lambda_\alpha} \right| \ = \ O \left( \frac{1}{\eta^3} + \frac{N^{B_V}}{\eta^{4}} + 1 \right).
\end{align}
For $z \in \mathbf{D}_2$, we note the first two drift terms are decreasing as $\eta$ increases, and the third drift term is $z$-independent. Thus, we may assume $\eta = 10$, in which case we have the bound
\begin{align}
\left| \frac{1}{N} \sum_{\lambda_\alpha \not\equiv 0} \ \frac{1}{(\lambda_\alpha - z)^3} \right| \ + \ \left| \frac{1}{2N} \sum_{\lambda_\alpha \not\equiv \pm \lambda_\beta} \ \frac{\lambda_\alpha + \lambda_\beta - 2z}{(\lambda_\alpha - z)^2(\lambda_\beta - z)^2} \right| \ + \ \left| \frac{M-N}{\lambda_\alpha} \right| \ \leq \ O \left(1 + N^{5B_V} \right).
\end{align}
Here, we use the a priori repulsion from the origin in bounding the singular eigenvalue term $\frac{M-N}{\lambda_{\alpha}}$. This provides the appropriate bound for the drift terms.

It remains to bound the diffusion terms with high probability. Because the diffusion terms are bounded, adapted and measurable, by Proposition \ref{prop:concentration} we have the following estimate uniformly over $z \in \mathbf{D}_{L,q}$ with $(\xi, \nu)$-high probability:
\begin{align}
\sup_{t \in [t_0, t_1]} \ \left| \sum_{\lambda_\alpha \not\equiv 0} \ \frac{1}{\sqrt{N}} \int_{t_0}^{t} \ \frac{1}{(\lambda_\alpha - z)^2} \ \d B_\alpha(t) \right| \ \leq \ C \E \left( \ \int_{t_0}^{t_1} \ \frac{1}{(\lambda_\alpha - z)^2} \ \d B_\alpha(t) \right)^2.
\end{align}
We compute the RHS with Ito's $L^2$-isometry to obtain the following bound for the martingale term:
\begin{align}
\sup_{t \in [t_0, t_1]} \ \left| \sum_{\lambda_\alpha \not\equiv 0} \ \frac{1}{\sqrt{N}} \int_{t_0}^{t} \ \frac{1}{(\lambda_\alpha - z)^2} \ \d B_\alpha(t) \right| \ = \ O \left( \frac{t_1 - t_0}{\eta^4} \right). 
\end{align}
\end{proof}
We now turn to deriving similar estimates for the Stieltjes transform of the free convolution measure. Because this measure is deterministic, we may use PDE techniques directly.
\begin{prop} \label{prop:STFCshorttimesdeformedmatrices}
Fix $\delta > 0$ and suppose $0 < t_1 < t_2$ are times satisfying 
\begin{align}
t_1 + t_2 \ \leq \ C N^{-\delta}
\end{align}
for some fixed constant $C > 0$. Suppose $V$ is $(g, G)$-regular at $E_0$, and let $0 < q < 1$. Then, uniformly over $z \in \mathbf{D}_{L,q}$, we have
\begin{align}
\left| m_{\on{fc}, t_2}(z) - m_{\on{fc}, t_1}(z) \right| \ = \ O \left( (t_2 - t_1) \left( \frac{1}{\eta} + \frac{1}{\eta^2}\right) \right).
\end{align}
\end{prop}
\begin{proof}
We briefly note here that the Stieltjes transform solves the following PDE:
\begin{align}
\partial_t \ m_{\on{fc},t}(z) \ = \ \frac12 \partial_z \left[ m_{\on{fc},t}(z) \left(m_{\on{fc},t}(z) + z \right) \right].
\end{align}
The desired estimate follows now from standard arguments, e.g. see Section 7 in \cite{LY}.
\end{proof}
With Lemma \ref{lemma:SDESTdeformedmatrices} and Proposition \ref{prop:STFCshorttimesdeformedmatrices}, we may now deduce a local law uniformly in time. We first partition the time interval $\mathscr{T}_{\omega}$ into a set of time intervals with small gaps as follows:
\begin{align}
\mathscr{T}_{\omega} \ = \ \mathscr{T}_{\omega, 1} \cup \ldots \cup \mathscr{T}_{\omega, K}: \ \sup_{1 \leq j \leq K} \ |\mathscr{T}_{\omega, j}| \ \leq \ N^{-5B_V - 4},  \quad \sup \mathscr{T}_{\omega, i} \ \leq \ \inf \mathscr{T}_{\omega, i+1}.
\end{align}
We note that, if $B_V$ is fixed, that $K = O(1)$. With $(\xi, \nu)$-high probability, by the pointwise (in time) local law in Theorem \ref{theorem:stronglocallawstronglocallawdeformedmatrices}, the strong local law holds for the $N^{O(1)}$ set of times $\{ \inf \mathscr{T}_{\omega,j} \}$. For notational convenience, we establish the following:
\begin{align}
t_j \ := \ \inf \mathscr{T}_{\omega, j}.
\end{align}
On the other hand, by Lemma \ref{lemma:SDESTdeformedmatrices} and Proposition \ref{prop:STFCshorttimesdeformedmatrices}, we have the following estimates for any fixed $j \in [[1, K]]$:
\begin{align}
\sup_{t \in \mathscr{T}_{\omega, j}} \ \sup_{z \in \mathbf{D}_{L,q}} \ |m_N(z; t) - m_{\on{fc},t}(z)| \ &\leq \ \sup_{t \in \mathscr{T}_{\omega, j}} \ \sup_{z \in \mathscr{D}_{L,q}} \ \left| m_N(z; t) - m_N(z; t_j) \right| \nonumber \\
&\quad \ + \ \sup_{t \in \mathscr{T}_{\omega, j}} \ \sup_{z \in \mathscr{D}_{L,q}} \ \left| m_N(z; t_j) - m_{\on{fc},t_j}(z) \right| \nonumber \\
&\quad \ + \ \sup_{t \in \mathscr{T}_{\omega, j}} \ \sup_{z \in \mathscr{D}_{L,q}} \ \left| m_{\on{fc}, t_j}(z) - m_{\on{fc}, t}(z) \right| \\
&\leq \ O \left( N^{-4} \left(1 + \frac{1}{\eta} + \frac{1}{\eta^2} + \frac{1}{\eta^3} + \frac{1}{\eta^4} \right) + \frac{N^{\e}}{N \eta} \right).
\end{align}
This completes the stochastic continuity argument and thus extends the strong local law to all times in $\mathscr{T}_{\omega}$ simultaneously with high probability.
%
%
%


\end{document}